\documentclass[11pt]{amsart}

\usepackage{amsmath,amsthm,amscd,amsfonts,wasysym,amssymb,epic,eepic,bbm,mathtools,multicol,enumitem,mathrsfs,comment}
\usepackage{tikz}
\usepackage{tikz-cd}
\usetikzlibrary{arrows.meta,calc,cd,math}
\usepackage[pagebackref,colorlinks=true,linkcolor=blue,citecolor=blue]{hyperref}
\usepackage[noabbrev]{cleveref}
\usepackage{MnSymbol}

\newcounter{dummy}
\makeatletter
\newcommand\myitem[1][]{\item[#1]\refstepcounter{dummy}\def\@currentlabel{#1}}
\makeatother

\allowdisplaybreaks

\newtheorem{thm}{Theorem}[section]
\newtheorem{cor}[thm]{Corollary}

\newtheorem*{thm*}{Theorem}
\newtheorem*{prop*}{Proposition}
\newtheorem{prop}[thm]{Proposition}
\theoremstyle{remark}
\newtheorem{remark}[thm]{Remark}

\newcounter{remarkscounter}

\newcommand{\A}{\mathbb{A}}
\newcommand{\C}{\mathbf{C}}

\newcommand{\Z}{\mathbb{Z}}

\newcommand{\Ad}{\mathrm{Ad}}
\newcommand{\Adj}{\mathrm{Adj}}

\newcommand{\GL}{\mathrm{GL}}

\newcommand{\Hom}{\mathrm{Hom}}
\newcommand{\Id}{\mathrm{Id}}

\newcommand{\Lie}{\mathrm{Lie}}

\newcommand{\Mat}{\mathrm{Mat}}

\newcommand{\SL}{\mathrm{SL}}

\newcommand{\Spec}{\mathrm{Spec}}

\newcommand{\Stab}{\mathrm{Stab}}

\newcommand{\Tr}{\mathrm{Tr}}

\renewcommand{\O}{\mathrm{O}}

\renewcommand{\dim}{\mathrm{dim}}

\renewcommand{\ker}{\mathrm{ker}\,}

\newcommand{\tr}{\mathrm{tr}}

\makeatletter
\newcommand{\rprod}{%
  \DOTSB             
  \mathop{\mathpalette\rprod@\relax}\slimits@}
\newcommand{\rprod@}[2]{%
  \ooalign{$\m@th#1\prod$\cr$\m@th#1\coprod$\cr}}
\makeatother

\newcommand{\quash}[1]{}

\theoremstyle{definition}
\newtheorem{defn}[thm]{Definition}
\newtheorem{lemma}[thm]{Lemma}

\newenvironment{psmatrix}
  {\left(\begin{smallmatrix}}
  {\end{smallmatrix}\right)}

\numberwithin{equation}{subsection}

\allowdisplaybreaks

\begin{document}
\linespread{1.2}

\title{Quasi-Classical Braverman--Kazhdan Intertwiners via Quiver Varieties}

\author{Nikolay Grantcharov, Aaron Slipper}
\address{Department of Mathematics\\
Duke University\\
Durham, NC 27708}
\email{aaron.slipper@duke.edu}
\address{Department of Mathematics\\
University of Georgia\\
Athens, GA 30607}
\email{nikolayg@uga.edu}

\begin{abstract}
We show that Braverman--Kazhdan normalized intertwiners for $G=\SL_n(\C)$ have a quasi-classical incarnation governed by type $A$ quiver varieties. More precisely, for standard parabolic subgroups $P$ and $P'$ with conjugate Levi subgroups, we construct $\SL_n\times L^{\mathrm{ab}}$-equivariant isomorphisms $\Phi(P,P'):\overline{T^*(\SL_n/[P,P])}^{\mathrm{aff}}\rightarrow\overline{T^*(\SL_n/[P',P'])}^{\mathrm{aff}}$ between the affinizations of the cotangent bundles of the corresponding Braverman--Kazhdan spaces, and we prove that these isomorphisms satisfy Coxeter relations. The construction uses $\SL$-gauge analogues of Lusztig--Maffei--Nakajima reflection functors, thereby extending Wang's quiver-variety realization of the quasi-classical Gelfand--Graev action from the Borel case to arbitrary parabolic subgroups. In this way, we complete the quasi-classical Braverman--Kazhdan intertwiner story for $\SL_n(\C)$ and obtain a systematic source of non-isomorphic varieties whose affinized cotangent bundles are isomorphic.
\end{abstract}

\maketitle

\setcounter{tocdepth}{1}
\tableofcontents

\section{Introduction}

Let $G$ be a reductive group, $W$ its Weyl group, and $U\subset G$ a maximal unipotent subgroup. Although the basic affine space $G/U$ does not itself carry a natural $W$-action, several related objects \textit{do} admit such a structure. The Gelfand--Graev action is an action of $W$ on the algebra of differential operators $D(G/U)$.\footnote{The original note on the Gelfand--Graev action appears to be unpublished; see \cite{GinzburgKazhdan2022}.} The ``Horospherical transforms" \cite{GelfandGraevPiatetskiiShapiro1969} are an action of $W$ on $L^2(G/U)$ (see also \cite{BK:basic:affine}). Finally, the quasi-classical Gelfand-Graev action is an action of $W$ on the Poisson algebra $\C[T^*(G/U)]$ \cite{GinzburgRiche2015,GinzburgKazhdan2022}. 

These three incarnations of the Gelfand--Graev action are best viewed not as a symmetry of the variety $G/U$, but as a Fourier-type symmetry of its differential, analytic, and symplectic avatars. In the simplest case $G=\SL_2$, where $G/U\simeq \A^2\setminus\{0\}$, these three avatars are (respectively) the classical Fourier involution on the Weyl algebra of $D_{\A^2}$, the symplectic Fourier transform on $L^2(\A^2)$, and the classical symplectic Fourier involution on $\A^4$.

Braverman and Kazhdan extended this picture from basic affine space to the spaces $G/[P,P]$, where $P$ is a parabolic subgroup of $G$. More precisely, they constructed normalized intertwining operators between the spaces $L^2(G/[P,P])$ as $P$ ranges over the parabolic subgroups of $G$ with a fixed Levi subgroup \cite{BK:normalized,Getz:Liu:BK}. The Gelfand--Graev action is recovered in the Borel case $P=B$, while the general parabolic case gives a larger system of intertwiners attached to varying Braverman--Kazhdan spaces $G/[P,P]$. These operators are part of the broader circle of ideas relating generalized Fourier transforms, relative trace formulas, and Langlands functoriality.

The present paper investigates the quasi-classical form of these normalized intertwiners for $G=\SL_n$, over the ground field $\C$. Our starting point is recent work of Wang \cite{Wang2021GelfandGraev}. In type $A$, Wang showed that the quasi-classical Gelfand--Graev action on $\overline{T^*(\SL_n/U)}^{\mathrm{aff}}$ can be realized geometrically by modifying the reflection functors of Lusztig, Maffei, and Nakajima \cite{Lusztig2000QuiverWeyl,Maffei2002RemarkQuiverWeyl,Nakajima2003ReflectionFunctors} on quiver varieties: the usual $\GL_{\mathbf{v}}$-gauge group is replaced by an $\SL_{\mathbf{v}}$-gauge group. This gives a striking bridge between two a priori different subjects: the Gelfand--Graev action and the geometry of Nakajima quiver varieties.\footnote{Gannon and Williams have also given a Coulomb-branch interpretation of both the quasi-classical and differential-operator Gelfand--Graev actions \cite{GannonWilliamsCoulomb}.}

Our main result extends Wang's construction from the Gelfand--Graev action in the Borel case to the full system of Braverman--Kazhdan intertwiners for $\SL_n$ attached to arbitrary associate parabolic subgroups $P$. More precisely, let $P$ and $P'$ be standard parabolic subgroups of $\SL_n$ with conjugate Levi subgroups $L$ and $L'$. We construct an $\SL_n\times L^{\mathrm{ab}}$-equivariant\footnote{The assumption of conjugate Levi subgroups identifies their abelianizations $L^{\mathrm{ab}}:=L/[L,L]$.} isomorphism
\begin{equation*}
\Phi(P,P'):\overline{T^*(\SL_n/[P,P])}^{\mathrm{aff}}\rightarrow\overline{T^*(\SL_n/[P',P'])}^{\mathrm{aff}}
\end{equation*}
between the affinizations of the cotangent bundles of the Braverman--Kazhdan spaces $\SL_n/[P,P]$ and $\SL_n/[P',P']$. Furthermore, these isomorphisms satisfy Coxeter relations and are therefore the quasi-classical analogues of the normalized Braverman--Kazhdan intertwiners. We note that \textit{a priori} it is not even obvious that the two affinized cotangent bundles are abstractly isomorphic;\footnote{This is particularly striking because the underlying Braverman--Kazhdan spaces $\SL_n/[P,P]$ are \textit{not} isomorphic; see Section~\ref{sectionNonIso}.} this fact is a corollary of our construction of $\Phi$. 

Extending Wang's construction to general parabolics is highly nontrivial. Wang specializes to the quiver with dimension vector $\mathbf{v}=(1,2,\ldots,n-1)$ and framing vector $\mathbf{w}=(0,\ldots,0,n)$. The Weyl group preserves the dimension vector $\mathbf{v}$, giving rise to an action on a single variety. We work with \textit{arbitrary} dimension vectors $\mathbf{v}=(d_1,\ldots,d_{k-1})$ such that $0<d_1<\cdots<d_{k-1}<n$, and with framing vector $\mathbf{w}=(0,\ldots,0,n)$. Reflections typically alter our dimension vectors.

This leads to several complications. First, Wang introduced a torsion-1 constraint which allowed him to modify the usual LMN $\GL_{\mathbf{v}}$-reflection functors to $\SL_{\mathbf{v}}$-quiver varieties. However, when the dimension vector is arbitrary, the torsion-1 constraint causes Coxeter relations to break. We introduce a new \textit{signed} torsion convention, depending on the dimension vector and the reflection, which recovers the Coxeter relations and gives rise to well-defined compositions of reflection functors.

Second, and more importantly, Wang critically utilizes several pleasant algebro-geometric properties of the quiver variety with dimension vector $\mathbf{v} = (1, 2, \ldots, n-1)$. These do \textit{not} hold in our more general setting. In particular, the $0$-fiber of the moment map\footnote{Often called the ``shell'' \cite{HerbigSchwarzSeaton2024} in the literature.} is not normal; the quotient map to the Hamiltonian reduction is not flat; and the resulting Hamiltonian reductions are not necessarily normal\footnote{This is in marked contrast with the case of a $\GL_{\mathbf{v}}$-gauge group, where normality is guaranteed by the central result of \cite{CrawleyBoevey2003Normality}.} (see Section~\ref{ExamplesSection}). Most devastatingly: the full $\SL_{\mathbf{v}}$-Hamiltonian reductions for reflected quivers are in general non-isomorphic varieties (see Appendix~\ref{1323ex}). Thus reflection functors between affinized cotangent bundles cannot arise from naive isomorphisms between the associated quiver varieties. To circumvent these obstacles, we construct our reflection functors as isomorphisms between the affinizations of suitable dense open subsets of the two Hamiltonian reductions.\footnote{Importantly, the codimension of the complement of these open subsets can be $1$, precluding any Hartogs-type argument, even when the Hamiltonian reductions are normal varieties; see Appendix~\ref{1323ex}.}

The main results of the paper are summarized below.

\begin{thm*}[Theorem~\ref{classical reflection functors}]
Let $P$ be a standard parabolic subgroup of $\SL_n$ whose Levi subgroup $L$ has $k$ blocks. Let $w\in S_k$, and let $P^w$ be the corresponding standard parabolic subgroup with Levi subgroup $L^w:=w\cdot L$, obtained by permuting the blocks of $L$ by $w$. Then there exists an isomorphism of varieties
\begin{equation*}
\Phi_w(P,P^w):\overline{T^*(\SL_n/[P,P])}^{\mathrm{aff}}\rightarrow\overline{T^*(\SL_n/[P^w,P^w])}^{\mathrm{aff}}.
\end{equation*}
\end{thm*}

The next result records the equivariance property expected of a quasi-classical Braverman--Kazhdan intertwiner.

\begin{prop*}[Propositions~\ref{WT equivariance} and~\ref{SL_n equivariance}]
Let $(P,P^w)$ be a pair of associate parabolics as in Theorem~\ref{classical reflection functors}. Then the reflection functor $\Phi_w:=\Phi_w(P,P^w)$ is equivariant with respect to the standard $\SL_n\times L^{\mathrm{ab}}$-action (Equation~\ref{SL_n times L^ab action}): for $[B]\in\overline{T^*(\SL_n/[P,P])}^{\mathrm{aff}}$,
\begin{align*}
\Phi_w([B]\cdot h)&=\Phi_w([B])\cdot w(h), \qquad \text{for all } h\in L^{\mathrm{ab}},\\
\Phi_w(g\cdot[B])&=g\cdot\Phi_w([B]), \qquad \text{for all } g\in\SL_n.
\end{align*}
\end{prop*}

Finally, the reflection functors are normalized: they satisfy the Coxeter relations.

\begin{thm*}[Theorem~\ref{braid relation theorem}]
Let $\sigma_i,\sigma_j,\sigma_r\in S_k$ denote simple reflections with $j=i+1$ and $r$ arbitrary. Then the reflection functors satisfy the $S_k$ Coxeter relations:
\begin{align*}
\Phi_i(P^{\sigma_i},P)\Phi_i(P,P^{\sigma_i})&=\mathrm{Id},\\
\Phi_r(P^{\sigma_i},P^{\sigma_r\sigma_i})\Phi_i(P,P^{\sigma_i})&=\Phi_i(P^{\sigma_r},P^{\sigma_i\sigma_r})\Phi_r(P,P^{\sigma_r}),\qquad \text{(if } |i-r|>1\text{)},\\
\Phi_i(P^{\sigma_j\sigma_i},P^{\sigma_i\sigma_j\sigma_i})\Phi_j(P^{\sigma_i},P^{\sigma_j\sigma_i})\Phi_i(P,P^{\sigma_i})
&=
\Phi_j(P^{\sigma_i\sigma_j},P^{\sigma_j\sigma_i\sigma_j})\Phi_i(P^{\sigma_j},P^{\sigma_i\sigma_j})\Phi_j(P,P^{\sigma_j}).
\end{align*}
\end{thm*}

These results complete the quasi-classical Braverman--Kazhdan intertwiner story for $G=\SL_n$ and make the relationship with quiver varieties completely explicit in type $A$. Moreover, they also provide a systematic source of examples of non-isomorphic varieties whose affinized cotangent bundles are isomorphic. We study this phenomenon in Section~\ref{sectionNonIso}.

This is the quasi-classical analogue of the question of when two non-isomorphic varieties can have isomorphic algebras of Grothendieck differential operators. Examples of $D$-isomorphic varieties have heretofore been somewhat \textit{ad hoc}: in dimension $1$, the only such examples come from curves homeomorphic to $\A^1$ with cusps at some finite collection of points; see \cite{BerestWilson2004Differential}. In higher dimension, the only known example of this phenomenon\footnote{Except for those trivially built from the previous examples.} was due to Levasseur, Smith, and Stafford (LSS), and is representation-theoretic in nature \cite[Section~9]{LevasseurSmithStafford1989Joseph}.\footnote{Morita equivalences of categories of $\mathcal D_X$-modules were studied by Ben-Zvi and Nevins in \cite{BenZviNevinsCuspsDModules}.} As we show in Section~\ref{quadricconex}, the reflection functors in the example of the quadric cone produce the quasi-classical analogue of LSS's mysterious isomorphism. Thus, our results explain this apparent coincidence and suggest a geometric mechanism for constructing further examples.

We close this introduction by placing the construction in a broader context. The question of when two varieties have Poisson-isomorphic cotangent bundles is closely related to the question of polarizations of a symplectic variety. In the relative Langlands program of Ben-Zvi, Sakellaridis, and Venkatesh, such polarizations are expected to give nonlinear analogues of the Fourier transform \cite[Section~10]{BenZviSakellaridisVenkatesh2024}, which are in turn expected to play an important role in Langlands functoriality \cite{Ngo:Hankel}. The Gelfand--Graev action and the Braverman--Kazhdan intertwiners are among the central representation-theoretic examples of this philosophy. Schematically,
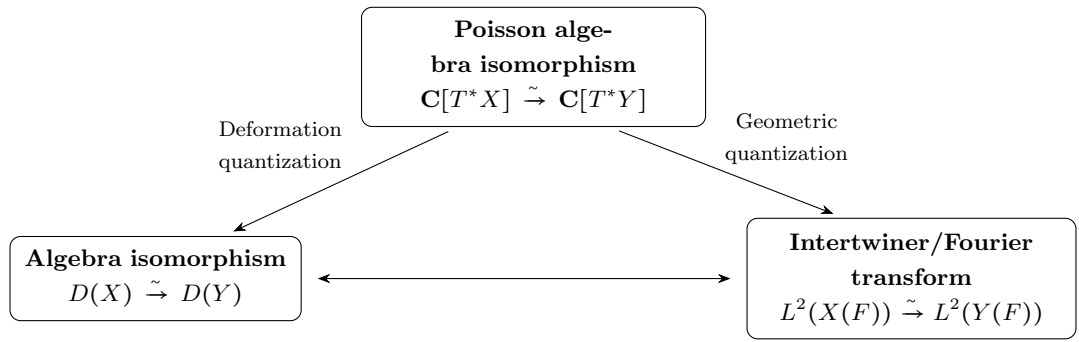
\[
\begin{tikzpicture}[
    >=Stealth,
    font=\footnotesize,
    box/.style={
        draw,
        rounded corners=4pt,
        align=center,
        inner sep=5pt
    },
    lab/.style={
        font=\scriptsize,
        fill=white,
        inner sep=1pt,
        align=center
    }
]

\node[box, text width=4.2cm] (P) at (0,0) {%
\textbf{Poisson algebra isomorphism}\\
$\C[T^*X]\xrightarrow{\sim}\C[T^*Y]$
};

\node[box, text width=3.5cm] (D) at (-5.0,-2.8) {%
\textbf{Algebra isomorphism}\\
$D(X)\xrightarrow{\sim}D(Y)$
};

\node[box, text width=4.0cm] (L) at (5.0,-2.8) {%
\textbf{Intertwiner/Fourier transform}\\
$L^2(X(F))\xrightarrow{\sim}L^2(Y(F))$
};

\draw[->, shorten <=4pt, shorten >=4pt]
    ($(P.south)+(-1.0,0)$) --
    ($(D.north)+(0.9,0)$)
    node[pos=.47, above left=1pt, lab] {Deformation\\quantization};

\draw[->, shorten <=4pt, shorten >=4pt]
    ($(P.south)+(1.0,0)$) --
    ($(L.north)+(-0.9,0)$)
    node[pos=.47, above right=1pt, lab] {Geometric\\quantization};

\draw[<->, shorten <=6pt, shorten >=6pt]
    (D.east) -- (L.west);

\end{tikzpicture}
\]
\noindent where $F$ is a local field. Conjectural generalizations of these ideas have appeared in \cite{GetzEtAl_ModulationGroups_2025}.

The differential-operator side of this picture is beginning to emerge for Braverman--Kazhdan spaces. In very recent work \cite{Hsu2026WeylAlgebrasBKSpaces}, Hsu has constructed, for $P$ a maximal parabolic, the corresponding intertwiner\footnote{Called the ``Fourier transform'' in loc. cit.} $D(G/[P,P])\rightarrow D(G/[P^{\mathrm{op}},P^{\mathrm{op}}])$ and established the basic theory of Weyl algebras on $G/[P,P]$. In our type $A$ setting, this corresponds to a quiver with two nodes, and it represents the quantization of our reflection functor. Thus the results of this paper provide the quasi-classical foundation for the expected quantized Braverman--Kazhdan intertwiners in type $A$.

Finally, we mention a related parabolic analogue of the Gelfand--Graev action, in which one studies intertwiners for $\overline{T^*(G/U_P)}^{\mathrm{aff}}$ rather than $\overline{T^*(G/[P,P])}^{\mathrm{aff}}$.\footnote{The dichotomy between the quasi-affine varieties $G/[P,P]$ and $G/U_P$ and their singular affinizations also arises in the theory of geometric Eisenstein series \cite{LaumonEisenstein,BravermanGaitsgoryGeometricEisenstein}, with the two forms of Drinfeld's compactification \cite{BFGMDrinfeldCompactificationsIC,BFGMDrinfeldCompactificationsICErratum}. In fact, Laumon's compactification, which is a small resolution of Drinfeld's compactification, may be thought of as a global function-field version of the quiver resolution $V\rightarrow \overline{\SL_n/[P,P]}^{\mathrm{aff}}$; see Corollary~\ref{generalized base affine quiver} and \cite{LaumonEisenstein,Kuznetsov1997LaumonResolution}. This explains why the intertwiners we construct here are analogous to the \textit{normalized} geometric Eisenstein series of the Geometric Langlands Program.} Tom Gannon has pointed out to us that these intertwiners may be constructed using the work of Gannon and Ginzburg \cite{GannonCotangentGUP,GinzburgDerivedSatakeSymplecticDuality} for general $G$, and \cite{GannonWebsterFunctorialityCoulomb} for $G=\SL_n$. This places the present construction inside a broader program relating Braverman--Kazhdan intertwiners, quiver varieties, Coulomb branches, and nonlinear Fourier transforms.

\subsection*{Acknowledgments}

We are grateful to Tom Gannon, Jayce Getz, Victor Ginzburg, Bill Graham, Boming Jia, Daniil Klyuev, Zhilin Luo, Dan Nakano and Xiangsheng Wang for numerous helpful discussions. We are especially thankful to Travis Schedler for pointing out that quiver varieties with an $\SL$-gauge group need not be normal. 

In preparing this paper, we used ChatGPT 5.5 for proofreading and running CAS tests on our calculations, especially in the Examples section. We also used it to find a solution, checked in Proposition \ref{torsion choice satisfies braid signs}, to Equations~\eqref{First torsion equation} and~\eqref{Second torsion equation}, thus completing the verification of Coxeter relations. All text is human written, and we assume full responsibility for its content.\footnote{In this AI acknowledgment, we follow recent practice; see, for example, \cite{TaoLocalBernsteinLebesgue}.}

\section{Reflection Functors for DKS Quiver Varieties}

In this section, we generalize Wang's construction \cite{Wang2021GelfandGraev} of reflection functors for quiver varieties with an $\SL$-gauge group, dimension vector $(1, 2, \ldots, n-1)$ and framing vector $(0, \ldots, 0, n)$, to type $A$ quivers whose dimension vector forms an arbitrary strictly increasing sequence. In other words, we remove Wang's assumption that $\mathbf{w}=C\mathbf{v}$ for $C$ the Cartan matrix. In this more general setting, the reflection functors are no longer isomorphisms between the corresponding Nakajima quiver varieties, since these varieties are not isomorphic in general; see Example~\ref{1323ex} for a simple counterexample. Instead, we obtain isomorphisms over smaller smooth open subsets, and then extend these to isomorphisms of their affinizations. In some cases, the complements of these open subsets have codimension only $1$ in the corresponding Nakajima quiver variety, so the Hartogs-principle/codimension-$\geq 2$ argument used to extend such isomorphisms to the full quiver varieties does not apply here, unlike in Wang's setting. We call the affinizations obtained in this way \textit{DKS quiver varieties}; see Definition~\ref{DKS quiver variety definition} for the precise construction.

Consider a dimension vector $\mathbf{v}=(d_1,\dots,d_{k-1})$ and framing vector $\mathbf{w}=(0,\dots,0,d_k)$ with $d_1<d_2<\cdots<d_k=n$. Let $Q$ be the quiver
\begin{equation}\label{DKS quiver}
Q=
\begin{tikzcd}
\overset{d_1}{\bullet} \arrow[r] & \overset{d_2}{\bullet} \arrow[r] & \cdots \arrow[r] & \overset{d_{k-1}}{\bullet} \arrow[r] & {\overset{d_k}{\boxed{\vphantom{d_1}\,\bullet\,}}}
\end{tikzcd}
\end{equation}
For simplicity of notation, we will denote such a quiver by the diagram
\[
Q=(d_1)-(d_2)-\cdots-(d_{k-1})-[d_k].
\]

Let
\[
V:=\bigoplus_{i=1}^{k-1}\Hom(\C^{d_i},\C^{d_{i+1}}),\qquad T^*V=V\times V^*,
\]
where we identify
\[
V^*=\bigoplus_{i=1}^{k-1}\Hom(\C^{d_{i+1}},\C^{d_i})
\]
using the trace pairing. Let
\[
\SL_{\mathbf{v}}=\prod_{i=1}^{k-1}\SL_{d_i}
\]
be the gauge group. The gauge group acts on $V$ and $V^*$ as follows. Given $(\alpha_i)\in V$, $(\beta_i)\in V^*$, and $(g_i)\in\SL_{\mathbf{v}}$, we define
\begin{equation*}
(g_1,\dots,g_{k-1})\cdot(\alpha_1,\dots,\alpha_{k-1})
=
(g_2\circ\alpha_1\circ g_1^{-1},\, g_3\circ\alpha_2\circ g_2^{-1},\,\dots,\, \alpha_{k-1}\circ g_{k-1}^{-1})
\end{equation*}
and
\begin{equation*}
(g_1,\dots,g_{k-1})\cdot(\beta_1,\dots,\beta_{k-1})
=
(g_1\circ\beta_1\circ g_2^{-1},\,\dots,\, g_{k-1}\circ\beta_{k-1}).
\end{equation*}

Define the moment map
\[
m:T^*V\rightarrow\mathfrak{sl}_{\mathbf{v}}^*=\bigoplus_{i=1}^{k-1}\mathfrak{sl}_{d_i}^*.
\]

\begin{equation}\label{sl quiver moment map}
(\alpha_i,\beta_i)\mapsto \left((X_i)\mapsto \sum_{i=1}^{k-1}\Tr\big((\alpha_{i-1}\beta_{i-1}-\beta_i\alpha_i)X_i\big)\right),
\end{equation}
where by convention $\alpha_0=\beta_0=0$.

We identify $\mathfrak{sl}_{\mathbf{v}}^*$ with $\mathfrak{sl}_{\mathbf{v}}$ using the trace pairing. The $0$-fiber $m^{-1}(0)$ may be identified with the set of $(\alpha_i,\beta_i)\in V\times V^*$ such that, for all $i$, there exists $\lambda_i\in\C$ with
\begin{equation}\label{moment map DKS}
\alpha_{i-1}\beta_{i-1}-\beta_i\alpha_i=\lambda_i\Id_{d_i}.
\end{equation}

\begin{defn}\label{oriented fiber}
Let $m:T^*V\rightarrow\mathfrak{sl}_{\mathbf{v}}$ be the moment map as in Equation~\ref{sl quiver moment map}, after identifying $\mathfrak{sl}_{\mathbf{v}}^*$ with $\mathfrak{sl}_{\mathbf{v}}$ using the trace pairing. Define the following open subset of the $0$-fiber:
\begin{equation}
m_{\mathrm{inj}}^{-1}(0):=\{(\alpha_i,\beta_i)\in m^{-1}(0): \text{each $\alpha_i$ is injective}\}.
\end{equation}
\end{defn}
 
Observe that $m_{\mathrm{inj}}^{-1}(0)$ is a smooth open subset of $m^{-1}(0)$. Indeed, the differential of $m$ is surjective along the locus where all $\alpha_i$ are injective: if $(X_i)\in\bigoplus_i\mathfrak{sl}_{d_i}$ annihilates the image of the differential, then varying $\beta_i$ gives $X_{i+1}\alpha_i-\alpha_iX_i=0$ for $1\leq i\leq k-2$, while varying $\beta_{k-1}$ gives $\alpha_{k-1}X_{k-1}=0$. Since all $\alpha_i$ are injective, descending induction gives $X_{k-1}=\cdots=X_1=0$. Thus the differential $d m$ has full rank on the injective locus, and $m_{\mathrm{inj}}^{-1}(0)$ is smooth. 

Moreover, by \cite[Lemma~5.29]{DancerKirwanSwann2013}, each point of $m_{\text{inj}}^{-1}(0)$ has trivial $SL_{\mathbf{v}}$-stabilizer.  Taking the corresponding geometric quotient, we obtain a smooth, not necessarily affine, variety
\begin{equation}\label{DKS Quiver Formula}
m_{\text{inj}}^{-1}(0)/SL_{\mathbf{v}}.
\end{equation}

\begin{defn}\label{DKS quiver variety definition}
We call the affinization of $m_{\text{inj}}^{-1}(0)/SL_{\mathbf{v}}$ \eqref{DKS Quiver Formula} a \textit{DKS Quiver variety}. It will be denoted by 
\begin{equation}
    m^{-1}_{\text{inj}}(0)//\SL_{\mathbf{v}}:=\Spec(\C[m^{-1}_{\text{inj}(0)}]^{\SL_{\mathbf{v}}}).
\end{equation}
\end{defn}
\noindent In general, we also denote the affine closure of a variety $X$ by $\overline{X}^{\text{aff}}=\Spec\, (\Gamma(X,\mathcal{O}_X)).$ Thus, by definition, 
$$\overline{m^{-1}_{\text{inj}}(0)/\SL_{\mathbf{v}}}^{\text{aff}} =m^{-1}_{\text{inj}}(0)//\SL_{\mathbf{v}}.$$

We caution that we do \textit{not} work with the full GIT quotient
\[
m^{-1}(0)//\SL_{\mathbf{v}}:=\Spec(\C[m^{-1}(0)]^{\SL_{\mathbf{v}}}).
\]
Indeed, these full affine quotients are not even isomorphic as abstract varieties in general when the dimension vector is transformed by a Weyl-group element $\sigma\in W$; see Appendix~\ref{1323ex} for an example. Moreover, quiver varieties with an $\SL$-gauge group can be quite pathological: in particular, they are not necessarily normal; see Section~\ref{NonnormalQuiverSection}. Thus, the conjecture following \cite[Definition~4.3]{Wang2021GelfandGraev} is unfortunately not true.

\subsection{The DKS Isomorphism}

Let us now recall the concrete description, from \cite{DancerKirwanSwann2013}, of the quotient $m_{\mathrm{inj}}^{-1}(0)/\SL_{\mathbf{v}}$.

Let $V^\circ:=\{(\alpha_i)\in V:\text{each $\alpha_i$ is injective}\}$. By a standard parabolic subgroup, we mean a parabolic subgroup of $\SL_n$ containing the Borel subgroup of upper triangular matrices.

\begin{prop}\label{oriented basic affine}
For a quiver $Q=(d_1)-\dots-[d_k]$ with $0=d_0<d_1<\dots<d_k=n$, there is an isomorphism
\begin{equation}
    \Psi:V^\circ/\SL_{\mathbf{v}}\rightarrow \SL_n/[P,P].
\end{equation}
Here $P$ is the standard parabolic subgroup whose Levi subgroup has block sizes, in order, $s_1,s_2,\dots,s_k$, where $s_i:=d_i-d_{i-1}$.
\end{prop}

\begin{proof}
This appears as \cite[Proposition~4.10]{DancerKirwanSwann2013}. We recall the proof for the reader's convenience.

Let
\[
\alpha_i^\circ:=
\left(\begin{array}{c}
\Id_{d_i\times d_i} \\
\hline
0_{s_{i+1}\times d_i}
\end{array}\right),
\]
where $s_{i+1}=d_{i+1}-d_i$. Thus $\alpha_i^\circ:\C^{d_i}\rightarrow\C^{d_{i+1}}$ is the standard inclusion.

Suppose $(\alpha_1,\dots,\alpha_{k-1})\in V^\circ$. Since each $\alpha_i$ has maximal rank and $d_i<d_{i+1}$, there exists $(g_1,\dots,g_k)\in \prod_{i=1}^k \SL_{d_i}=\SL_{\mathbf{v}}\times\SL_n$ such that
\[
(g_i)_{i=1}^k\cdot(\alpha_1,\dots,\alpha_{k-1})=(\alpha_1^\circ,\dots,\alpha_{k-1}^\circ).
\]
Indeed, starting with any $g_1\in\SL_{d_1}$, the map $\alpha_i(v)\mapsto\alpha_i^\circ(g_i v)$ from $\alpha_i(\C^{d_i})$ to $\alpha_i^\circ(\C^{d_i})$ can be extended to an element $g_{i+1}\in\SL_{d_{i+1}}$, since the complementary block can be chosen with arbitrary determinant.

Next, consider the stabilizer
\[
H:=\Stab_{\SL_{\mathbf{v}}\times\SL_n}(\alpha_1^\circ,\dots,\alpha_{k-1}^\circ).
\]
We claim that projection to the final factor identifies $H$ with $[P,P]$. Indeed, the stabilizer condition at the $i$th arrow says that $g_{i+1}$ has the following block form:
\[
g_{i+1}=
\left(\begin{array}{c|c}
g_i & * \\
\hline
0 & h_i
\end{array}\right),
\]
where $h_i\in\SL_{s_{i+1}}$. Iterating this description, the final component $g_k$ is block upper triangular with diagonal blocks of determinant $1$. Hence $g_k\in[P,P]$. Conversely, every element of $[P,P]$ arises uniquely in this way by reading off the nested diagonal subblocks. Thus projection to the final factor gives an identification $H\simeq[P,P]$.

Now choose $(g_i)_{i=1}^k$ such that $(g_i)_{i=1}^k\cdot(\alpha_i)_{i=1}^{k-1}=(\alpha_i^\circ)_{i=1}^{k-1}$. If $(g_i')_{i=1}^k$ is another such choice, then $(g_i'g_i^{-1})_{i=1}^k\in H$, so $g_k'^{-1}$ and $g_k^{-1}$ differ by right multiplication by an element of $[P,P]$. Therefore the formula
\[
(\alpha_i)_{i=1}^{k-1}\mapsto g_k^{-1}[P,P]
\]
descends to a well-defined map $\Psi:V^\circ/\SL_{\mathbf{v}}\rightarrow \SL_n/[P,P]$.

The inverse sends $g[P,P]$ to the $\SL_{\mathbf{v}}$-orbit of $(1,\dots,1,g)\cdot(\alpha_1^\circ,\dots,\alpha_{k-1}^\circ)$. The stabilizer computation above shows that this is well-defined and inverse to $\Psi$.
\end{proof}

\begin{cor}\label{generalized base affine quiver}
Assume $d_i<d_{i+1}$ for all $i$. Then
\begin{equation}
V//\SL_{\mathbf{v}}:=\Spec(\C[V]^{\SL_{\mathbf{v}}})\simeq \overline{\SL_n/[P,P]}^{\mathrm{aff}}.
\end{equation}
\end{cor}

\begin{proof}
By Proposition~\ref{oriented basic affine}, $V^\circ/\SL_{\mathbf{v}}\simeq \SL_n/[P,P]$. Therefore the affinization
\[
V^\circ//\SL_{\mathbf{v}}:=\Spec\big(\Gamma(V^\circ,\mathcal{O}_{V^\circ})^{\SL_{\mathbf{v}}}\big)
\]
is isomorphic to $\overline{\SL_n/[P,P]}^{\mathrm{aff}}$. It remains to show that $V//\SL_{\mathbf{v}}\simeq V^\circ//\SL_{\mathbf{v}}$.

The affine space $V$ is smooth, hence normal. Moreover, $V\setminus V^\circ$ is the union of the rank-deficient loci
\[
Z_i:=\{(\alpha_j)\in V:\operatorname{rank}(\alpha_i)\leq d_i-1\}.
\]
The locus where $\alpha_i$ fails to have maximal rank has codimension $d_{i+1}-d_i+1$. Since $d_i<d_{i+1}$, this codimension is at least $2$. Hence $\operatorname{codim}(V\setminus V^\circ)\geq 2$. By the algebraic Hartogs theorem, restriction induces an isomorphism $\C[V]\simeq \Gamma(V^\circ,\mathcal{O}_{V^\circ})$. The claim now follows by taking $\SL_{\mathbf{v}}$-invariants.
\end{proof}

The following result was also proven in \cite{DancerKirwanSwann2013}. For the reader's convenience, we sketch the proof, since we will use the construction of the map in establishing Proposition~\ref{DKS equivariance}.

\begin{thm}[\cite{DancerKirwanSwann2013}]\label{DKS quiver variety}
For $Q$ as in Equation~\ref{DKS quiver}, there is an isomorphism
\begin{equation}
    \Xi:m_{\mathrm{inj}}^{-1}(0)/\SL_{\mathbf{v}}\rightarrow T^*(\SL_n/[P,P]),
\end{equation}
where $P$ is the standard parabolic subgroup as in Proposition~\ref{oriented basic affine}.
\end{thm}

\begin{proof}[\textit{Proof sketch}]
We canonically identify
\begin{equation*}
T^*(\SL_n/[P,P])=\SL_n\times_{[P,P]}\Lie([P,P])^\perp
\end{equation*}
and $\Lie([P,P])^\perp=\mathfrak{u}_{P}\oplus\mathfrak{z}(\mathfrak{l})$. Here, $\mathfrak{u}_P$ is the nilradical of $\mathfrak{p}=\Lie(P)$ and $\mathfrak{z}(\mathfrak{l})$ is the center of the Levi subalgebra $\mathfrak{l}=\Lie(L)$, which may be identified with $\C^{k-1}$, where $k$ is the number of blocks of $L\subset \SL_n$.

Next, because $(\alpha_i,\beta_i)\in m_{\mathrm{inj}}^{-1}(0)$, we may find $(g_i)_{i=1}^{k}\in\prod_{i=1}^k \SL_{d_i}$ such that $g_{i+1}\alpha_i g_i^{-1}=\alpha_i^\circ$. Replacing $(\alpha_i,\beta_i)$ by this representative and suppressing tildes, we may assume $\alpha_i=\alpha_i^\circ$ for all $i$. By the moment map equations \eqref{moment map DKS}, we then have
\begin{equation*}
\alpha_i\beta_i=
\left(
\begin{array}{c|c}
\alpha_{i-1}\beta_{i-1}-\lambda_i\Id_{d_i} & D_{d_i\times s_{i+1}} \\
\hline
0_{s_{i+1}\times d_i} & 0_{s_{i+1}\times s_{i+1}}
\end{array}
\right),
\end{equation*}
where the block decomposition is $\C^{d_{i+1}}=\C^{d_i}\oplus\C^{s_{i+1}}$ and $D_{d_i\times s_{i+1}}$ is an arbitrary matrix. It follows by induction that $X:=g_k\alpha_{k-1}\beta_{k-1}g_k^{-1}$ lies in $\mathfrak{z}(\mathfrak{l})\oplus\mathfrak{u}_P$, up to subtracting its scalar trace part. Finally, given $X\in\mathfrak{gl}_n$, define its trace-free part 

\begin{equation}\label{TraceFreeGamma}
    \Gamma(X):=X-\frac{\Tr(X)}{n}\Id_n\in\mathfrak{sl}_n.
\end{equation}
The DKS map is then given by
\begin{align}\label{explicit DKS}
    \Xi:m_{\mathrm{inj}}^{-1}(0)/\SL_{\mathbf{v}}&\rightarrow \SL_n\times_{[P,P]}\Lie([P,P])^\perp\\
    [(\alpha_1,\beta_1,\dots,\alpha_{k-1},\beta_{k-1})] &\mapsto [g_k^{-1}, \Gamma(g_k\alpha_{k-1}\beta_{k-1}g_k^{-1})]. \qedhere
\end{align}
\end{proof}

As a result, we obtain the following corollary.

\begin{cor}\label{DKSAffinizedIso}
For any quiver $Q=(d_1)-\dots-[d_k]$ with $d_1<\dots<d_k=n$, we have an isomorphism
\[
m_{\mathrm{inj}}^{-1}(0)//\SL_{\mathbf{v}}:=\overline{m_{\mathrm{inj}}^{-1}(0)/\SL_{\mathbf{v}}}^{\mathrm{aff}}\simeq\overline{T^*(\SL_n/[P,P])}^{\mathrm{aff}}.
\]
\end{cor}

We emphasize that, in general, we do not have $m^{-1}(0)//\SL_{\mathbf{v}}=m_{\mathrm{inj}}^{-1}(0)//\SL_{\mathbf{v}}$, even though this equality holds in the full-rank case $\mathbf{v}=(1,2,\dots,n-1)$ \cite[Theorem~7.18]{DancerKirwanSwann2013}. See Appendix~\ref{1323ex} for a discussion of the counterexample when $Q=(2)-[3]$.

Let us now establish various equivariance properties of the DKS map $\Xi$. There is a natural $\SL_n\times L/[L,L]$-action on $\SL_n\times_{[P,P]}\Lie([P,P])^\perp$ given by the formula
\begin{equation}\label{SL_n times L^ab action}
(g,t)\cdot[g',x]:=[gg't,\Ad_{t^{-1}}(x)].
\end{equation}
Let us now transport this action to the DKS quiver varieties.

\begin{defn}\label{SLn action quiver variety}
Suppose $g\in\SL_n$ and $[\alpha_1,\beta_1,\dots,\alpha_{k-1},\beta_{k-1}]\in m_{\mathrm{inj}}^{-1}(0)/\SL_{\mathbf{v}}$. Then the $\SL_n$-action on $m_{\mathrm{inj}}^{-1}(0)/\SL_{\mathbf{v}}$ is given by
\begin{equation}\label{SL_n action on quiver}
g\cdot[\alpha_1,\beta_1,\dots,\alpha_{k-1},\beta_{k-1}]:=[\alpha_1,\beta_1,\dots,g\alpha_{k-1},\beta_{k-1}g^{-1}].
\end{equation}
\end{defn}

\noindent The $\SL_n$-action on $m_{\mathrm{inj}}^{-1}(0)/\SL_{\mathbf{v}}$ is induced by the standard action on the framed vertex $\C^n$. Thus the DKS map $\Xi$ is $\SL_n$-equivariant:
\begin{align*}
    \Xi(g\cdot[\alpha_1,\beta_1,\dots,\alpha_{k-1},\beta_{k-1}])
    &=\Xi([\alpha_1,\beta_1,\dots,g\alpha_{k-1},\beta_{k-1}g^{-1}])\\
    &=[gg_k^{-1},\Gamma((g_k g^{-1})(g\alpha_{k-1}\beta_{k-1}g^{-1})(gg_k^{-1}))]\\
    &=[gg_k^{-1},\Gamma(g_k\alpha_{k-1}\beta_{k-1}g_k^{-1})]\\
    &=g\cdot\Xi([\alpha_1,\beta_1,\dots,\alpha_{k-1},\beta_{k-1}]).
\end{align*}

Next, we address the $L^{\mathrm{ab}}:=L/[L,L]$ compatibility. We represent this torus by the block-scalar subgroup $Z(L)^\circ\subset L$.

\begin{defn}\label{L^ab action quiver variety}
Define the $L^{\mathrm{ab}}\simeq Z(L)^\circ$-action on $m^{-1}_{\mathrm{inj}}(0)/\SL_{\mathbf{v}}$ as follows: Let
\[
h=\mathrm{diag}(z_1,\dots,z_1,z_2,\dots,z_2,\dots,z_k,\dots,z_k)\in Z(L)^\circ,
\]
where $z_i$ appears $d_i-d_{i-1}$ times, $d_0:=0$, and $\prod_i z_i^{d_i-d_{i-1}}=1$. Define
\[
h_i:=\mathrm{diag}(z_1,\dots,z_1,\dots,z_i,\dots,z_i)\in\GL_{d_i},
\]
where $z_a$ appears $d_a-d_{a-1}$ times. Then the action of $h$ on $[\alpha_1,\beta_1,\dots,\alpha_{k-1},\beta_{k-1}]\in m_{\mathrm{inj}}^{-1}(0)/\SL_{\mathbf{v}}$ is given by
\begin{equation}\label{L/[L,L] action on quiver}
\begin{aligned}
\alpha_i &\mapsto h_{i+1}^{-1}\alpha_i h_i,
&
\beta_i &\mapsto h_i^{-1}\beta_i h_{i+1},
\qquad 1\leq i\leq k-2,\\
\alpha_{k-1} &\mapsto \alpha_{k-1}h_{k-1},
&
\beta_{k-1} &\mapsto h_{k-1}^{-1}\beta_{k-1}.
\end{aligned}
\end{equation}
\end{defn}

Now we check $L/[L,L]$-equivariance of $\Xi$. Observe first that if $g_k\alpha_{k-1}=\alpha_{k-1}^\circ$, then the identity $h\alpha_{k-1}^\circ=\alpha_{k-1}^\circ h_{k-1}$ implies $(h^{-1}g_k)(\alpha_{k-1}h_{k-1})=\alpha_{k-1}^\circ$. Thus, writing $\Gamma(X)=X-\frac{\Tr(X)}{n}\Id_n$ \eqref{TraceFreeGamma} for the trace-free projection, we have
\[
\begin{aligned}
\Xi(B\cdot h)
&=\left[g_k^{-1}h,
\Gamma\!\left((h^{-1}g_k)(\alpha_{k-1}h_{k-1})(h_{k-1}^{-1}\beta_{k-1})(g_k^{-1}h)\right)\right]\\
&=\left[g_k^{-1}h,
\Gamma(h^{-1}g_k\alpha_{k-1}\beta_{k-1}g_k^{-1}h)\right] \\
&=
\left[g_k^{-1}h,
h^{-1}\Gamma(g_k\alpha_{k-1}\beta_{k-1}g_k^{-1})h\right] \\
&=\Xi(B)\cdot h.
\end{aligned}
\]

Next, define
\[
c_i:=\det(h_i)=\prod_{a=1}^i z_a^{d_a-d_{a-1}}.
\]
Then, after choosing a $d_i$th root of $c_i$, we have $\widetilde h_i:=c_i^{-1/d_i}h_i\in\SL_{d_i}$. Since $(\widetilde h_1,\dots,\widetilde h_{k-1})\in\SL_{\mathbf v}$, replacing $h_i$ by $\widetilde h_i$ in \eqref{L/[L,L] action on quiver} gives the identity map on the quotient $m_{\mathrm{inj}}^{-1}(0)/\SL_{\mathbf v}$. Hence the $L/[L,L]$-action can equivalently be represented by the following scalar action:
\begin{equation}\label{modified L/[L,L] action}
\begin{aligned}
\alpha_i &\mapsto c_{i+1}^{-1/d_{i+1}}c_i^{1/d_i}\alpha_i,
&
\beta_i &\mapsto c_i^{-1/d_i}c_{i+1}^{1/d_{i+1}}\beta_i,
\qquad 1\leq i\leq k-2,\\
\alpha_{k-1} &\mapsto c_{k-1}^{1/d_{k-1}}\alpha_{k-1},
&
\beta_{k-1} &\mapsto c_{k-1}^{-1/d_{k-1}}\beta_{k-1}.
\end{aligned}
\end{equation}
The resulting action on the quotient is independent of the choices of roots. We summarize this discussion with the following proposition.

\begin{prop}\label{DKS equivariance}
The DKS isomorphism $\Xi:m_{\mathrm{inj}}^{-1}(0)/\SL_{\mathbf v}\rightarrow T^*(\SL_n/[P,P])$ is $\SL_n\times L/[L,L]$-equivariant, where the action on $m_{\mathrm{inj}}^{-1}(0)/\SL_{\mathbf v}$ is defined in Equations~\ref{SL_n action on quiver} and~\ref{L/[L,L] action on quiver}.
\end{prop}

\subsection{The reflection functors}

The main goal is to prove the following theorem, which generalizes Wang's main result to the parabolic case.

\begin{thm}\label{classical reflection functors}
Let $P$ be a standard parabolic subgroup of $\SL_n$ whose Levi subgroup $L$ has $k$ blocks. Let $w\in S_k$, and let $P^w$ be the corresponding standard parabolic subgroup whose Levi subgroup $L^w$ is obtained by permuting the blocks of $L$ by $w$. Then there exists an isomorphism of varieties
\[
\Phi_w(P,P^w):\overline{T^*(\SL_n/[P,P])}^{\mathrm{aff}}\rightarrow\overline{T^*(\SL_n/[P^w,P^w])}^{\mathrm{aff}}.
\]
Furthermore, by conjugating the parabolics $P$ and $P^w$, we may remove the assumption that they are standard.
\end{thm}

Let us first define the reflection functors for simple reflections $\sigma_i\in S_k$.

\begin{defn}
Fix an index $i$, with $1\leq i\leq k-1$. Let $P$ be the standard parabolic whose Levi subgroup has block sizes $s_j:=d_j-d_{j-1}$, where $0=d_0<d_1<d_2<\dots<d_k=n$. Define $d_i':=d_{i-1}+d_{i+1}-d_i$ and $d_j':=d_j$ for $j\neq i$. Then $0=d_0'<d_1'<\dots<d_k'=n$. Let $P'$ be the standard parabolic whose Levi subgroup has block sizes $s_j':=d_j'-d_{j-1}'$. We call such a pair of parabolics $(P,P')$ \textit{$i$-swapped}.
\end{defn}

Thus, for an $i$-swapped pair of parabolics $(P,P')$, we have $s_i'=s_{i+1}$ and $s_{i+1}'=s_i$, while $s_j'=s_j$ for $j\neq i,i+1$. Thus $P$ and $P'$ are standard parabolics whose Levi subgroups are obtained from one another by swapping the $i$th and $(i+1)$st blocks.

Corresponding to the pair of parabolics $(P,P')$, consider the induced moment maps on quiver varieties:
\[
m_P:T^*V\rightarrow\mathfrak{sl}_{\mathbf{v}}^*,\qquad m_{P'}:T^*V'\rightarrow\mathfrak{sl}_{\mathbf{v}'}^*.
\] 

Let $B=(\alpha_i,\beta_i)_i\in m_{\mathrm{inj}}^{-1}(0)$ denote an arbitrary element. Fix the convention $\C^{d_0}=0$ and $\alpha_0=\beta_0=0$. For $1\leq i\leq k-1$, define the following two linear maps:
\begin{align}
\operatorname{out}_i(B)&:=
\left(\begin{smallmatrix}
\beta_{i-1}\\
\alpha_i
\end{smallmatrix}\right)
:\C^{d_i}\rightarrow \C^{d_{i-1}+d_{i+1}},\\
\operatorname{in}_i(B)&:=
\left(\begin{smallmatrix}
-\alpha_{i-1} & \beta_i
\end{smallmatrix}\right)
:\C^{d_{i-1}+d_{i+1}}\rightarrow\C^{d_i}.
\end{align}
In order to define the reflection functors, we must define the following two open subsets of $m_{\mathrm{inj}}^{-1}(0)$.

\begin{defn}
Under the notation in the preceding paragraph, define
\begin{align}
\Lambda_i^{\mathrm{inj}}&:=\{B\in m_{\mathrm{inj}}^{-1}(0): \operatorname{out}_i(B) \text{ is injective}\}, \\
\Lambda_i^{\mathrm{surj}}&:=\{B\in m_{\mathrm{inj}}^{-1}(0): \operatorname{in}_i(B) \text{ is surjective}\}.
\end{align}
Since $\alpha_i$ is injective for all $i$ on $m_{\mathrm{inj}}^{-1}(0)$, the map $\operatorname{out}_i(B)$ is automatically injective. Hence $\Lambda_i^{\mathrm{inj}}=m_{\mathrm{inj}}^{-1}(0)$. When we wish to specify the parabolic $P$ which defines the moment map, we write $\Lambda_{P,i}^{\mathrm{inj}}$ and $\Lambda_{P,i}^{\mathrm{surj}}$.
\end{defn}

Here is the main definition.

\begin{defn}\label{zipper}
Fix an index $i$, with $1\leq i\leq k-1$. Let $(P,P')$ be a pair of $i$-swapped parabolics. Fix a nonzero scalar $\tau_i:=\tau_i(P,P')\in\C^*$ depending on $i$, $P$, and $P'$. Let $B=(\alpha_j,\beta_j)_{j=1}^{k-1}$ denote an element of $m_{P,\mathrm{inj}}^{-1}(0)$, and similarly let $B'=(\alpha_j',\beta_j')_{j=1}^{k-1}\in m_{P',\mathrm{inj}}^{-1}(0)$.

Let
\[
\mu_i(B):=\alpha_{i-1}\beta_{i-1}-\beta_i\alpha_i\in\mathfrak{gl}_{d_i}
\]
denote the $i$th component of the $\mathfrak{gl}_{\mathbf v}$-valued moment map. Since $B\in m_P^{-1}(0)$, we have $\mu_i(B)=\lambda_i(B)\Id_{d_i}$ for some scalar $\lambda_i(B)\in\C$.

Define
\[
Z_i:=Z_i^{\tau_i}(P,P')\subset \Lambda_{P,i}^{\mathrm{surj}}\times\Lambda_{P',i}^{\mathrm{inj}}
\]
to be the subvariety consisting of pairs $(B,B')\in \Lambda_{P,i}^{\mathrm{surj}}\times\Lambda_{P',i}^{\mathrm{inj}}$ satisfying the following conditions:

\begin{enumerate}
    \item[(C1)] $\alpha_j=\alpha_j'$ and $\beta_j=\beta_j'$ for all $j\notin\{i-1,i\}$.

    \item[(C2)] The following sequence is short exact:
    \begin{equation}\label{SES zipper}
        0\rightarrow \C^{d_i'}\xrightarrow{\operatorname{out}_i(B')}\C^{d_{i-1}}\oplus\C^{d_{i+1}}\xrightarrow{\operatorname{in}_i(B)}\C^{d_i}\rightarrow 0,
    \end{equation}
    where
    \[
    \operatorname{out}_i(B')=
    \begin{psmatrix}
        \beta_{i-1}'\\
        \alpha_i'
    \end{psmatrix},
    \qquad
    \operatorname{in}_i(B)=
    \begin{psmatrix}
        -\alpha_{i-1} & \beta_i
    \end{psmatrix}.
    \]
    Furthermore, the short exact sequence \eqref{SES zipper} has torsion $\tau_i\in\C^*$ in the sense of \cite[Ch.~1]{Turaev2002Torsions}. Concretely, this means that if we fix a volume form $\mathrm{vol}_j\in\bigwedge^j\C^j$ for each $j$, then
    \[
    \operatorname{out}_i(B')(\mathrm{vol}_{d_i'})\wedge \operatorname{in}_i(B)^{-1}(\mathrm{vol}_{d_i})
    =
    \tau_i\,\mathrm{vol}_{d_{i-1}}\wedge\mathrm{vol}_{d_{i+1}},
    \]
    where $\operatorname{in}_i(B)^{-1}(\mathrm{vol}_{d_i})$ denotes any exterior lift of $\mathrm{vol}_{d_i}$ along $\operatorname{in}_i(B)$.

    \item[(C3)] The following compatibility with the moment map holds:
    \[
    \operatorname{out}_i(B')\operatorname{in}_i(B')
    =
    \operatorname{out}_i(B)\operatorname{in}_i(B)
    +\lambda_i(B)\Id_{\C^{d_{i-1}}\oplus\C^{d_{i+1}}}
    :
    \C^{d_{i-1}}\oplus\C^{d_{i+1}}\rightarrow \C^{d_{i-1}}\oplus\C^{d_{i+1}}.
    \]
\end{enumerate}
\end{defn}

\noindent We again emphasize that although the constructions of the $\Lambda_i$ and $Z_i$ look similar to those in Wang, it is crucial that we work only inside the smooth loci $m_{\mathrm{inj}}^{-1}(0)$. Wang sets $\tau_i(P,P')=1$ for all $i$, $P$, and $P'$, but, as we will see when verifying the braid relations, it is important to allow an arbitrary nonzero scalar.

The following diagram illustrates the maps appearing in the definition of $Z_i$. The condition (C2) says that the sequence formed in the middle by taking the upward-pointing arrows is short exact:
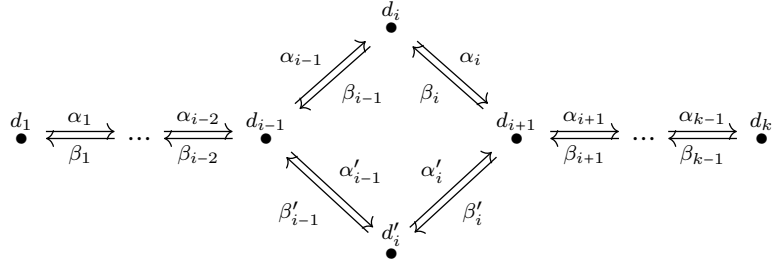
\[
\begin{tikzcd}
                                                       &                                                                   &                                                                                                                        & \overset{d_i}{\bullet} \arrow[rd, "\alpha_i", shift left] \arrow[ld, "\beta_{i-1}"]    &                                                                                                                       &                                                                   &                                             \\
\overset{d_1}{\bullet} \arrow[r, "\alpha_1", shift left] & \cdots \arrow[r, "\alpha_{i-2}", shift left] \arrow[l, "\beta_1"] & \overset{d_{i-1}}{\bullet} \arrow[ru, "\alpha_{i-1}", shift left] \arrow[rd, "\alpha_{i-1}'", shift left] \arrow[l, "\beta_{i-2}"] &                                                                                            & \overset{d_{i+1}}{\bullet} \arrow[r, "\alpha_{i+1}", shift left] \arrow[lu, "\beta_i"] \arrow[ld, "\beta_i'"] & \cdots \arrow[l, "\beta_{i+1}"] \arrow[r, "\alpha_{k-1}", shift left] & \overset{d_k}{\bullet} \arrow[l, "\beta_{k-1}"] \\
&& &\overset{d_i'}{\bullet} \arrow[ru, "\alpha_i'", shift left] \arrow[lu, "\beta_{i-1}'"] &    & &  
\end{tikzcd}
\]
At the endpoints, we interpret nonexistent arrows using the convention $\C^{d_0}=0$ and $\alpha_0=\beta_0=0$. We will refer to this diagram representing $Z_i$ as the \textit{$i$th zipper}. Note that there is a standard action of
\[
G_{i,\mathbf v}:=\SL_{d_i'}\times \SL_{d_i}\times \prod_{\substack{1\leq l\leq k-1\\ l\neq i}}\SL_{d_l}
\]
on $Z_i$. Also note that $m_{P,\mathrm{inj}}^{-1}(0)$ and $m_{P',\mathrm{inj}}^{-1}(0)$ are $G_{i,\mathbf v}$-invariant.

\noindent The following Proposition follows directly from \cite[Lemma~33]{Maffei2002RemarkQuiverWeyl} and \cite[Lemma~3.5.3]{Wang2021GelfandGraev}.
\begin{prop}\label{zipper torsor}  Let $(P,P')$ be a pair of $i$-swapped parabolics. The natural projection $p:Z_i\rightarrow \Lambda_{P,i}^{\mathrm{surj}}$ is a principal $\SL_{d_i'}$-bundle, and the projection $p':Z_i\rightarrow \Lambda_{P',i}^{\mathrm{inj}}$ is a principal $\SL_{d_i}$-bundle.
\end{prop}

\begin{proof}
We sketch the proof for $p:Z_i\rightarrow \Lambda_{P,i}^{\mathrm{surj}}$. Since $\operatorname{in}_i(B)$ is surjective, $\ker \operatorname{in}_i(B)$ has dimension $d_i'$. Thus exactness of \eqref{SES zipper} is equivalent to choosing an isomorphism $\C^{d_i'}\simeq \ker \operatorname{in}_i(B)$. Such choices form a $\GL_{d_i'}$-torsor, and imposing torsion $\tau_i$ cuts this down to an $\SL_{d_i'}$-torsor.

It remains only to note that, once $\operatorname{out}_i(B')$ has been chosen, condition (C3) uniquely determines $\operatorname{in}_i(B')$. Indeed, set
\[
T_i:=\C^{d_{i-1}}\oplus\C^{d_{i+1}}
\]
and
\[
A:=\operatorname{out}_i(B)\operatorname{in}_i(B)+\lambda_i(B)\Id_{T_i}.
\]
Since $\operatorname{in}_i(B)\operatorname{out}_i(B)=-\lambda_i(B)\Id_{\C^{d_i}}$, we have $\operatorname{in}_i(B)A=0$. Hence $\operatorname{Im}(A)\subseteq\ker \operatorname{in}_i(B)=\operatorname{Im}(\operatorname{out}_i(B'))$. Since $\operatorname{out}_i(B')$ is injective, there is a unique map $\operatorname{in}_i(B')$ such that
\[
\operatorname{out}_i(B')\operatorname{in}_i(B')=A.
\]
This is exactly condition (C3). The verification that the resulting $B'$ lies in $\Lambda_{P',i}^{\mathrm{inj}}$ and satisfies the moment-map equations is the same as in \cite[Lemma~33]{Maffei2002RemarkQuiverWeyl}. The proof for $p'$ is analogous.
\end{proof}

Thus, given a pair of $i$-swapped parabolics $(P,P')$, we have the following isomorphisms:
\[
\begin{tikzcd}
        &{\C[Z_i]^{G_{i,\mathbf{v}}}} &                                                                   \\
{\C[\Lambda_{P,i}^{\mathrm{surj}}]^{\SL_{\mathbf{v}}}} \arrow[ru, "p^*"] &                             & {\C[\Lambda_{P',i}^{\mathrm{inj}}]^{\SL_{\mathbf{v}'}}}. \arrow[lu, "(p')^*"']
\end{tikzcd}
\]

For each $i$, with $1\leq i\leq k-1$, define the isomorphism
\[
\psi_i^{\tau_i}:=(p^*)^{-1}(p')^*:\C[\Lambda_{P',i}^{\mathrm{inj}}]^{\SL_{\mathbf{v}'}}\rightarrow \C[\Lambda_{P,i}^{\mathrm{surj}}]^{\SL_{\mathbf{v}}}.
\]
We will often omit the dependence of $\psi_i^{\tau_i}$ on the torsion scalar $\tau_i(P,P')$ from the notation, since it plays no role until we check the braid relations.

\begin{prop}
Let $s_i:=d_i-d_{i-1}$. Then
\begin{equation*}
\operatorname{codim}_{m_{\mathrm{inj}}^{-1}(0)}\big(m_{\mathrm{inj}}^{-1}(0)\setminus\Lambda_i^{\mathrm{surj}}\big)
=
\begin{cases}
1 &\text{if } s_{i+1}<s_i,\\
2+s_{i+1}-s_i &\text{if } s_{i+1}\geq s_i.
\end{cases}
\end{equation*}
\end{prop}

\begin{proof}
Suppose $B:=(\alpha_j,\beta_j)\in m_{\mathrm{inj}}^{-1}(0)\setminus\Lambda_i^{\mathrm{surj}}$. Then
\[
\operatorname{in}_i(B)=
\begin{psmatrix}
-\alpha_{i-1} & \beta_i
\end{psmatrix}
:\C^{d_{i-1}}\oplus\C^{d_{i+1}}\rightarrow\C^{d_i}
\]
is not surjective. Since $\alpha_{i-1}$ is injective, $\operatorname{in}_i(B)$ is surjective if and only if the induced map
\[
\overline{\beta}_i:\C^{d_{i+1}}\rightarrow \C^{d_i}/\operatorname{Im}(\alpha_{i-1})
\]
is surjective. Let $\pi_i:\C^{d_i}\rightarrow\C^{d_i}/\operatorname{Im}(\alpha_{i-1})$ be the quotient map. Reducing the moment-map equation $\alpha_{i-1}\beta_{i-1}-\beta_i\alpha_i=\lambda_i(B)\Id_{d_i}$
modulo $\operatorname{Im}(\alpha_{i-1})$ gives
\[
-\overline{\beta}_i\alpha_i=\lambda_i(B)\pi_i.
\]
If $\lambda_i(B)\neq 0$, then $\overline{\beta}_i\alpha_i$ is surjective, and hence $\overline{\beta}_i$ is surjective. This contradicts the assumption that $\operatorname{in}_i(B)$ is not surjective. Therefore
\[
m_{\mathrm{inj}}^{-1}(0)\setminus\Lambda_i^{\mathrm{surj}}
\subseteq
\{\lambda_i=0\}\cap m_{\mathrm{inj}}^{-1}(0).
\]

On the locus $\lambda_i=0$, the relation above becomes $\overline{\beta}_i\alpha_i=0$. Hence $\overline{\beta}_i$ factors through a map
\[
\overline{\overline{\beta}}_i:\C^{d_{i+1}}/\operatorname{Im}(\alpha_i)\rightarrow \C^{d_i}/\operatorname{Im}(\alpha_{i-1}).
\]
Moreover, $B\notin\Lambda_i^{\mathrm{surj}}$ if and only if this induced map $\overline{\overline{\beta}}_i$ is not surjective.

This induced map is intrinsic as a map between quotient spaces, but not as a map into a fixed vector space. Since codimension is local, we now work on a standard trivializing chart. Choose local coordinates in which
\[
\alpha_{i-1}=
\begin{psmatrix}
\Id_{d_{i-1}}\\
0
\end{psmatrix}
:\C^{d_{i-1}}\rightarrow \C^{d_i},
\qquad
\alpha_i=
\begin{psmatrix}
\Id_{d_i}\\
0
\end{psmatrix}
:\C^{d_i}\rightarrow \C^{d_{i+1}}.
\]
Equivalently, write
\[
\C^{d_i}=\C^{d_{i-1}}\oplus W_i,
\qquad
\C^{d_{i+1}}=\C^{d_{i-1}}\oplus W_i\oplus W_{i+1},
\]
where $\dim W_i=s_i$ and $\dim W_{i+1}=s_{i+1}$. In these coordinates, write $\beta_i:\C^{d_{i+1}}\rightarrow\C^{d_i}$ in block form as
\[
\beta_i=
\begin{psmatrix}
B_{11} & B_{12} & B_{13}\\
B_{21} & B_{22} & B_{23}
\end{psmatrix}.
\]
Then the induced map $\overline{\overline{\beta}}_i:\C^{d_{i+1}}/\operatorname{Im}(\alpha_i)\rightarrow \C^{d_i}/\operatorname{Im}(\alpha_{i-1})$ is represented in this chart by the block
\[
B_{23}:W_{i+1}\rightarrow W_i.
\]

It remains to check that $B_{23}$ is a free variable after imposing the moment-map equations. Write
\[
\beta_{i-1}=
\begin{psmatrix}
C_{11} & C_{12}
\end{psmatrix}
:\C^{d_{i-1}}\oplus W_i\rightarrow \C^{d_{i-1}}.
\]
The moment-map equation $\alpha_{i-1}\beta_{i-1}-\beta_i\alpha_i=\lambda_i(B)\Id_{d_i}$
forces
\[
B_{11}=C_{11}-\lambda_i(B)\Id_{d_{i-1}},
\qquad
B_{12}=C_{12},
\qquad
B_{21}=0,
\qquad
B_{22}=-\lambda_i(B)\Id_{W_i}.
\]
In particular, on the locus $\lambda_i=0$ we have
\[
\beta_i=
\begin{psmatrix}
C_{11} & C_{12} & B_{13}\\
0 & 0 & B_{23}
\end{psmatrix}.
\]
Thus the equation at the $i$th vertex imposes no condition on $B_{23}$. If $i<k-1$, the next moment-map equation, after choosing the analogous standard chart for $\alpha_{i+1}$, gives
\[
\beta_{i+1}\alpha_{i+1}
=
\begin{psmatrix}
C_{11}-\lambda_{i+1}(B)\Id_{d_{i-1}} & C_{12} & B_{13}\\
0 & -\lambda_{i+1}(B)\Id_{W_i} & B_{23}\\
0 & 0 & -\lambda_{i+1}(B)\Id_{W_{i+1}}
\end{psmatrix}.
\]
This determines the corresponding block of $\beta_{i+1}$ but imposes no relation on $B_{23}$. If $i=k-1$, there is no next moment-map equation, so there is nothing further to check. The remaining moment-map equations do not involve $B_{23}$. Hence, on this chart, the following assignment is a projection:
\[
\pi_U:\{\lambda_i=0\}\cap m_{\mathrm{inj}}^{-1}(0)\cap U\rightarrow \Hom(W_{i+1},W_i),
\qquad
B\mapsto B_{23}.
\]

Let $D\subset \Hom(W_{i+1},W_i)$ be the determinantal locus of maps of rank at most $s_i-1$. On the chart $U$, the discussion above gives
\[
\big(m_{\mathrm{inj}}^{-1}(0)\setminus\Lambda_i^{\mathrm{surj}}\big)\cap U
=
\pi_U^{-1}(D).
\]
Here, the equality is understood inside $\{\lambda_i=0\}\cap m_{\mathrm{inj}}^{-1}(0)\cap U$. Since $\pi_U$ is a projection, we have
\[
\operatorname{codim}_{\{\lambda_i=0\}\cap m_{\mathrm{inj}}^{-1}(0)\cap U}\pi_U^{-1}(D)
=
\operatorname{codim}_{\Hom(W_{i+1},W_i)}D.
\]
The determinantal codimension is
\[
\operatorname{codim}_{\Hom(W_{i+1},W_i)}D
=
\begin{cases}
0 &\text{if } s_{i+1}<s_i,\\
s_{i+1}-s_i+1 &\text{if } s_{i+1}\geq s_i.
\end{cases}
\]
Finally, the same local coordinate calculation shows that $\lambda_i$ is a free coordinate, so the hypersurface $\{\lambda_i=0\}$ has codimension $1$ in $m_{\mathrm{inj}}^{-1}(0)$. Therefore
\[
\operatorname{codim}_{m_{\mathrm{inj}}^{-1}(0)}\big(m_{\mathrm{inj}}^{-1}(0)\setminus\Lambda_i^{\mathrm{surj}}\big)
=
1+\operatorname{codim}_{\Hom(W_{i+1},W_i)}D,
\]
which is precisely the stated formula.
\end{proof}

\begin{cor}\label{Hartogs for injective locus}
Suppose $s_i\leq s_{i+1}$. Then
\[
\C[m_{\mathrm{inj}}^{-1}(0)]=\C[\Lambda_i^{\mathrm{inj}}]
\qquad\text{and}\qquad
\C[m_{\mathrm{inj}}^{-1}(0)]=\C[\Lambda_i^{\mathrm{surj}}].
\]
\end{cor}

\begin{proof}
The first equality follows from the fact that $\Lambda_i^{\mathrm{inj}}=m_{\mathrm{inj}}^{-1}(0)$. For the second equality, recall that $m_{\mathrm{inj}}^{-1}(0)$ is smooth, hence normal. Since $s_i\leq s_{i+1}$, the previous proposition implies that the complement of $\Lambda_i^{\mathrm{surj}}$ inside $m_{\mathrm{inj}}^{-1}(0)$ has codimension at least $2$. Thus Hartogs' lemma gives
\[
\Gamma(m_{\mathrm{inj}}^{-1}(0),\mathcal{O})
=
\Gamma(\Lambda_i^{\mathrm{surj}},\mathcal{O}),
\]
which is the desired equality.
\end{proof}

Finally, we may give the main definition.

\begin{defn}\label{reflection functor on simples}
Suppose $(P,P')$ is a pair of $i$-swapped parabolics, and suppose that the $i$th and $(i+1)$st block sizes of $P$ satisfy $s_i\leq s_{i+1}$. Fix a torsion scalar $\tau_i\in\C^*$. Define the isomorphism
\[
\Phi_i:=\Phi_i^{\tau_i}(P,P'):\overline{T^*(\SL_n/[P,P])}^{\mathrm{aff}}\rightarrow \overline{T^*(\SL_n/[P',P'])}^{\mathrm{aff}}
\]
to be the unique isomorphism making the following diagram commute:
\[
\begin{tikzcd}
{\overline{T^*(\SL_n/[P,P])}^{\mathrm{aff}}} \arrow[rr, "{\Phi_i^{\tau_i}(P,P')}"] \arrow[d, "\mathrm{DKS}"', Rightarrow, equals]                         &                                                   & {\overline{T^*(\SL_n/[P',P'])}^{\mathrm{aff}}} \arrow[d, "\mathrm{DKS}"', Rightarrow, equals]                                   \\
{\operatorname{Spec}\big(\C[m_{P,\mathrm{inj}}^{-1}(0)]^{\SL_{\mathbf{v}}}\big)} \arrow[d, "\mathrm{Hartogs}"', Rightarrow, equals]         &                                                   & {\operatorname{Spec}\big(\C[m_{P',\mathrm{inj}}^{-1}(0)]^{\SL_{\mathbf{v}'}}\big)} \arrow[d, Rightarrow, equals] \\
{\operatorname{Spec}\big(\C[\Lambda_{P,i}^{\mathrm{surj}}]^{\SL_{\mathbf{v}}}\big)} \arrow[rr, "{\operatorname{Spec}(\psi_i^{\tau_i})}"] &                                                   & {\operatorname{Spec}\big(\C[\Lambda_{P',i}^{\mathrm{inj}}]^{\SL_{\mathbf{v}'}}\big)}         \\
&{\operatorname{Spec}\big(\C[Z_i]^{G_{i,\mathbf{v}}}\big)} \arrow[lu, "{\operatorname{Spec}(p^*)}"] \arrow[ru, "{\operatorname{Spec}((p')^*)}"'] &
\end{tikzcd}
\]
Here, the ``Hartogs'' identification is the result of applying Corollary \ref{Hartogs for injective locus}. 

In the case $s_i>s_{i+1}$, define
\[
\Phi_i^{\tau_i}(P,P'):=\big(\Phi_i^{\tau_i}(P',P)\big)^{-1},
\]
where the right-hand side is defined by the previous case for the reversed pair $(P',P)$.
\end{defn}

\begin{proof}[Proof of Theorem~\ref{classical reflection functors}]
Suppose the torsion scalars $\tau_i(P,P')$ appearing in the definition of the zippers $Z_i^{\tau_i}(P,P')$ are chosen as in Definition~\ref{torsion choice}. Let $w=\sigma_{i_n}\cdots\sigma_{i_1}$ be a reduced expression in $S_k$. Starting with $P_0:=P$, define recursively $P_r:=P_{r-1}^{\sigma_{i_r}}$ for $1\leq r\leq n$, so that $P_n=P^w$. We define $\Phi_w(P,P^w):=\Phi_{i_n}(P_{n-1},P_n)\circ\cdots\circ\Phi_{i_1}(P_0,P_1)$. By Theorem~\ref{braid relation theorem}, this is independent of the chosen reduced expression for $w$. Hence $\Phi_w(P,P^w)$ is a well-defined isomorphism $\overline{T^*(\SL_n/[P,P])}^{\mathrm{aff}}\rightarrow\overline{T^*(\SL_n/[P^w,P^w])}^{\mathrm{aff}}$.

Finally, if $P$ and $P^w$ are not standard, choose conjugates which are standard, construct the corresponding isomorphism in the standard case, and then conjugate back.
\end{proof}
\noindent Note, in the proof of Theorem~\ref{classical reflection functors}, we may actually choose an arbitrary torsion scalar $\tau_i(P,P')$ and obtain \textit{an} isomorphism $\Phi_w$ for each reduced expression of $w$. But, these do not compose well.

We conclude this section by establishing that the reflection functors are also $\SL_n\times L^{\mathrm{ab}}$-equivariant.

\begin{prop}\label{WT equivariance}
Let $(P,P')$ be a pair of $i$-swapped parabolics. Then the reflection functor $\Phi_i:=\Phi_i(P,P')$ satisfies the following ``twisted'' $L^{\mathrm{ab}}$-equivariance relation:
\begin{equation}\label{L^ab equivariance}
\Phi_i([B]\cdot h)=\Phi_i([B])\cdot\sigma_i(h),\qquad \text{for all } h\in L^{\mathrm{ab}}.
\end{equation}
Here, if $h$ is represented by a block-scalar element $\operatorname{diag}(z_1,\dots,z_1,\dots,z_k,\dots,z_k)$, then $\sigma_i(h)$ is obtained by permuting the $i$th and $(i+1)$st blocks of $h$. In particular, $\sigma_i$ defines an isomorphism $L^{\mathrm{ab}}\simeq (L')^{\mathrm{ab}}$, and we interpret the $L^{\mathrm{ab}}$- and $(L')^{\mathrm{ab}}$-actions in \eqref{L^ab equivariance} as in Equation~\ref{L/[L,L] action on quiver}.
\end{prop}

\begin{proof}
It suffices to prove the claim in the case $s_i\leq s_{i+1}$; the case $s_i>s_{i+1}$ follows by applying this case to the reversed pair $(P',P)$ and then taking inverses. We show that the correspondence $Z_i^{\tau_i}(P,P')$ is compatible with the torus actions. Thus, suppose $(B,B')\in Z_i^{\tau_i}(P,P')$. We will show that $(B\cdot h,B'\cdot\sigma_i(h))\in Z_i^{\tau_i}(P,P')$.

Choose a block-scalar representative $h=(z_1,\dots,z_1,\dots,z_k,\dots,z_k)$ for the $L^{\mathrm{ab}}$-action, and define $c_j:=z_1^{s_1}\cdots z_j^{s_j}$ and $q_j:=c_j^{1/d_j}c_{j+1}^{-1/d_{j+1}}$, with the convention $c_k=1$. Thus the scalar form of the action in Equation~\ref{modified L/[L,L] action} is $\alpha_j\mapsto q_j\alpha_j$ and $\beta_j\mapsto q_j^{-1}\beta_j$. Let $h':=\sigma_i(h)$, and let $c_j'$ and $q_j'$ be the corresponding scalars for the $P'$-side. Then $c_j'=c_j$ for $j\neq i$, while $c_i'=c_{i-1}z_{i+1}^{s_{i+1}}$. It follows that $q_j'=q_j$ for $j\notin\{i-1,i\}$ and $q_{i-1}q_i=q_{i-1}'q_i'$. Define $\rho:=q_i'/q_i=q_{i-1}/q_{i-1}'$.

Condition (C1) is now immediate, since all arrows away from the zipper are scaled by the same factors on both sides. For (C2) and (C3), set $T_i:=\C^{d_{i-1}}\oplus\C^{d_{i+1}}$ and let $S:T_i\rightarrow T_i$ be the scalar automorphism $S=q_{i-1}^{-1}\Id_{\C^{d_{i-1}}}\oplus q_i\Id_{\C^{d_{i+1}}}$. Then
\begin{equation}\label{torus action on in-out}
\operatorname{out}_i(B\cdot h)=S\operatorname{out}_i(B),\qquad
\operatorname{in}_i(B\cdot h)=\operatorname{in}_i(B)S^{-1},
\end{equation}
while
\begin{equation}\label{torus action on primed in-out}
\operatorname{out}_i(B'\cdot h')=\rho S\operatorname{out}_i(B'),\qquad
\operatorname{in}_i(B'\cdot h')=\rho^{-1}\operatorname{in}_i(B')S^{-1}.
\end{equation}
Since $\operatorname{in}_i(B)\operatorname{out}_i(B')=0$, Equations~\eqref{torus action on in-out} and~\eqref{torus action on primed in-out} imply that the transformed maps also compose to zero, and exactness is preserved because $S$ and $\rho$ are invertible.

It remains to check that the torsion scalar in (C2) is unchanged. Let $s:\C^{d_i}\rightarrow T_i$ be any splitting of $\operatorname{in}_i(B)$. Then $Ss$ is a splitting of $\operatorname{in}_i(B\cdot h)$. Hence the torsion expression is multiplied by $\rho^{d_i'}\det(S)$. Since $\det(S)=q_{i-1}^{-d_{i-1}}q_i^{d_{i+1}}$, a direct calculation using the definitions of the $q$'s gives
\begin{equation}\label{torsion scaling identity}
\rho^{d_i'}q_{i-1}^{-d_{i-1}}q_i^{d_{i+1}}=\frac{c_ic_i'}{c_{i-1}c_{i+1}}=1.
\end{equation}
Thus the torsion remains $\tau_i$, so condition (C2) is preserved.

Finally, condition (C3) is also preserved. Indeed, by Equations~\eqref{torus action on in-out} and~\eqref{torus action on primed in-out}, and using $\lambda_i(B\cdot h)=\lambda_i(B)$, we have
\begin{equation}\label{C3 torus equivariance calculation}
\begin{aligned}
\operatorname{out}_i(B'\cdot h')\operatorname{in}_i(B'\cdot h')
&=
S\operatorname{out}_i(B')\operatorname{in}_i(B')S^{-1}\\
&=
S\big(\operatorname{out}_i(B)\operatorname{in}_i(B)+\lambda_i(B)\Id_{T_i}\big)S^{-1}\\
&=
\operatorname{out}_i(B\cdot h)\operatorname{in}_i(B\cdot h)+\lambda_i(B\cdot h)\Id_{T_i}.
\end{aligned}
\end{equation}
Therefore condition (C3) is preserved. Hence $(B\cdot h,B'\cdot\sigma_i(h))\in Z_i^{\tau_i}(P,P')$, and the equivariance relation \eqref{L^ab equivariance} follows.
\end{proof}

\begin{prop}\label{SL_n equivariance}
The reflection functors $\Phi_i(P,P')$ are $\SL_n$-equivariant: for $g\in\SL_n$ and $[B]\in m_{\mathrm{inj}}^{-1}(0)//\SL_{\mathbf{v}}$,
\begin{equation}\label{SLn equivariance reflection functor}
\Phi_i(g\cdot[B])=g\cdot\Phi_i([B]),
\end{equation}
where the $\SL_n$-action is as in Definition~\ref{SL_n action on quiver}.
\end{prop}

\begin{proof}
It is enough to show that the correspondence $Z_i^{\tau_i}(P,P')$ is stable under the diagonal $\SL_n$-action. Suppose $(B,B')\in Z_i^{\tau_i}(P,P')$. We must show that $(g\cdot B,g\cdot B')\in Z_i^{\tau_i}(P,P')$.

If $1\leq i\leq k-2$, this is immediate: the $\SL_n$-action only changes the final arrow pair $(\alpha_{k-1},\beta_{k-1})$, while the conditions (C2) and (C3) defining the $i$th zipper involve only the arrows with indices $i-1$ and $i$. Condition (C1) is also preserved, since the final arrow pair is changed in the same way for $B$ and $B'$.

It remains to consider $i=k-1$. Set $T_{k-1}:=\C^{d_{k-2}}\oplus\C^n$ and let $S:T_{k-1}\rightarrow T_{k-1}$ be the automorphism $S=\Id_{\C^{d_{k-2}}}\oplus g$. Then
\begin{equation}\label{SLn action on last in-out}
\operatorname{out}_{k-1}(g\cdot B')=S\operatorname{out}_{k-1}(B'),\qquad
\operatorname{in}_{k-1}(g\cdot B)=\operatorname{in}_{k-1}(B)S^{-1}.
\end{equation}
Since $\operatorname{in}_{k-1}(B)\operatorname{out}_{k-1}(B')=0$, Equation~\eqref{SLn action on last in-out} implies $\operatorname{in}_{k-1}(g\cdot B)\operatorname{out}_{k-1}(g\cdot B')=0$. Exactness is preserved because $S$ is an isomorphism. The torsion is also preserved, since $\det(S)=\det(g)=1$.

It remains only to check (C3). Similarly to \eqref{SLn action on last in-out}, we have
\begin{equation}\label{SLn action on last unprimed and primed}
\operatorname{out}_{k-1}(g\cdot B)=S\operatorname{out}_{k-1}(B),\qquad
\operatorname{in}_{k-1}(g\cdot B')=\operatorname{in}_{k-1}(B')S^{-1}.
\end{equation}
Since $\lambda_{k-1}(g\cdot B)=\lambda_{k-1}(B)$, Equations~\eqref{SLn action on last in-out} and~\eqref{SLn action on last unprimed and primed} give
\begin{equation}\label{SLn C3 equivariance calculation}
\begin{aligned}
\operatorname{out}_{k-1}(g\cdot B')\operatorname{in}_{k-1}(g\cdot B')
&=S\operatorname{out}_{k-1}(B')\operatorname{in}_{k-1}(B')S^{-1}\\
&=S\big(\operatorname{out}_{k-1}(B)\operatorname{in}_{k-1}(B)+\lambda_{k-1}(B)\Id_{T_{k-1}}\big)S^{-1}\\
&=\operatorname{out}_{k-1}(g\cdot B)\operatorname{in}_{k-1}(g\cdot B)+\lambda_{k-1}(g\cdot B)\Id_{T_{k-1}}.
\end{aligned}
\end{equation}
Thus (C3) is preserved. Therefore $(g\cdot B,g\cdot B')\in Z_i^{\tau_i}(P,P')$, and Equation \eqref{SLn equivariance reflection functor} follows.
\end{proof}

\section{The Coxeter Relations}\label{BraidRelationSect}

The entirety of this section is devoted to proving the following theorem.
\begin{thm}\label{braid relation theorem}
Let $\sigma_i,\sigma_j,\sigma_r\in S_k$ denote the simple reflections and suppose $j=i+1$. Choose torsion scalars associated to each $i$- or $j$-swapped pair of parabolics as in Definition~\ref{torsion choice}. Then the reflection functors satisfy the $S_k$ Coxeter relations:
\begin{align}
\Phi_i(P,P^{\sigma_i})\Phi_i(P^{\sigma_i},P)
&=\mathrm{Id}
\label{quadratic relation}\\
\Phi_i(P,P^{\sigma_i})\Phi_r(P^{\sigma_i},P^{\sigma_r\sigma_i})
&=\Phi_i(P^{\sigma_r},P^{\sigma_i\sigma_r})\Phi_r(P,P^{\sigma_r}),
\qquad \text{if } |i-r|>1.
\label{easy braid}\\
\Phi_i(P^{\sigma_j\sigma_i},P^{\sigma_i\sigma_j\sigma_i})\Phi_j(P^{\sigma_i},P^{\sigma_j\sigma_i})\Phi_i(P,P^{\sigma_i})
&=
\Phi_j(P^{\sigma_i\sigma_j},P^{\sigma_j\sigma_i\sigma_j})\Phi_i(P^{\sigma_j},P^{\sigma_i\sigma_j})\Phi_j(P,P^{\sigma_j}),
\label{braid relation}
\end{align}
\end{thm}
\noindent We formulated the braid relation \eqref{braid relation} so that no inverse reflection functors appear. We may indeed arrange this: if the $i^{th}$, $(i+1)^{st}$, and $(i+2)^{nd}$ block sizes of $P$ are $s_i,s_{i+1},s_{i+2}$ with $s_i\leq s_{i+1}\leq s_{i+2}$, then both words $\sigma_i\sigma_j\sigma_i$ and $\sigma_j\sigma_i\sigma_j$ move a smaller block one step to the right at each stage.

\begin{lemma}
The quadratic relation  $\Phi_i(P,P^{\sigma_i})\Phi_i(P^{\sigma_i},P)=\mathrm{Id}$ holds.
\end{lemma}
\begin{proof}
    Recall that by definition, when $s_i>s_{i+1}$, we defined $\Phi_i(P,P^{\sigma_i}):=\Phi_i(P^{\sigma_i},P)^{-1}$. Thus the only non-trivial case is when $s_i=s_{i+1}$. This case then follows from Wang's proof, which we recall for the reader's convenience. Using the notation as in Definition \ref{zipper}, suppose $(B,B')\in Z_i^{\tau_i}(P,P^{\sigma_i})$. We wish to show $(B',B)\in Z_i^{\tau_i}(P^{\sigma_i},P)$. Conditions (C1) and (C3) trivially hold.
    
    Let us prove condition $(C2)$ holds for $(B',B)$. For simplicity, denote $a=\text{out}_i(B), a'=\text{out}_i(B'),b=\text{in}_i(B),b'=\text{in}_i(B')$. Also let $d_i=d_i'=d$. 
\noindent To check two morphisms are equal, it suffices to check on the dense subset (See Equation \ref{dense open Mi}) where $\lambda_i\neq0$. The moment map identities
\begin{equation*}
    ba=-\lambda_i\text{Id}\quad b'a'=\lambda_i\text{Id}
\end{equation*}
imply exactness for the sequence associated with $(B',B)$ \cite[Remark 36]{Maffei2002RemarkQuiverWeyl}. It remains to check torsion is preserved. For a chosen basis $\{x_1,\dots, x_{d}\}$ of $\C^{d}$ comptible with a fixed volume form, we have
\begin{equation*}
    b(a(-x_\alpha/\lambda_i))=x_\alpha,\quad b'(a'(x_\alpha/\lambda_i))=x_\alpha
\end{equation*}
\noindent Thus the torsion condition for the pair $(B,B')$ implies
\begin{equation*}
\bigwedge_{a=1}^{d}a'(x_\alpha)\wedge\bigwedge_{a=1}^da(-x_\alpha/\lambda_i)=\tau_i\text{vol}_{\C^{d_{i-1}+d_{i+1}}}.
\end{equation*}
\noindent For the sequence attached to the pair $(B',B)$, the torsion is $\bigwedge_{a=1}^{d}a(x_\alpha)\wedge\bigwedge_{a=1}^{d}a'(x_\alpha/\lambda_i)$ Thus we see this torsion is $(-1)^{d^2+d}\tau_i=\tau_i$ since $d^2+d$ is even. This completes the proof. \footnote{Note that we did not need to specialize $\tau_i$ to be as in Definition \ref{torsion choice}, and the torsion choice in Definition \ref{torsion choice} is symmetric about $(P,P')$ when $s_i=s_{i+1}$.}
\end{proof}

The relation \eqref{easy braid} is immediate, since the construction of the zipper at vertex $i$ does not affect the construction of the zipper at vertex $r$ when $|i-r|>1$.

Verifying the braid relation \eqref{braid relation} is quite technical; we first give a sketch of our strategy. In the case where all blocks of the parabolic have the same size, Wang verified the braid relations by running Maffei's argument while keeping careful track of the torsion of the short exact sequence in condition (C2) of Definition~\ref{zipper}. Central to Maffei's argument is the existence of a unique pair of maps which ``complete'' the hexagon diagram \eqref{braid-hexagon}. The crux of Wang's argument is establishing that the short exact sequence in (C2) associated to this pair has torsion $1$. This is done by keeping track of the torsion in the two linear algebra lemmas of \cite[Lemmas~40 and~41]{Maffei2002RemarkQuiverWeyl}.

Unfortunately, as Equations~\eqref{First torsion equation} and~\eqref{Second torsion equation} will show, in the general case where the parabolic $P$ has blocks of different sizes, the short exact sequence associated to Maffei's pair of maps need not have torsion $1$. To fix this, we impose in the definition of the reflection functor that the short exact sequence in (C2) has torsion $\tau_i(P,P')=(-1)^{\ell_i(P)}$, where $\ell_i(P)\in\Z/2\Z$ is determined by the block sizes of $P$ and the parity of $i$. Our choice of $\ell_i(P)$ is made in order to satisfy Equations~\eqref{First torsion equation} and~\eqref{Second torsion equation}. This choice is not unique, but it has the property that $\ell_i(P)=0$ if all blocks of $P$ have the same size. In this case, our proof is identical to Wang's.

We now introduce our notation. Since morphisms which agree on a dense open subset are equal, it suffices to prove Equation~\eqref{braid relation} on suitable dense open subsets. We choose these as follows. Given $B\in m_{P,\mathrm{inj}}^{-1}(0)$, let $\lambda_i(B)\in\C$ be the scalar appearing in the $i$th component of the moment map, as in Definition~\ref{zipper}. Let $\pi:m_{P,\mathrm{inj}}^{-1}(0)\rightarrow m_{P,\mathrm{inj}}^{-1}(0)//\SL_{\mathbf{v}}$ denote the natural projection. Define
\begin{equation}\label{dense open Mi}
\mathcal{M}_i:=\{\pi(B): B\in m_{P,\mathrm{inj}}^{-1}(0)\text{ and }\lambda_i(B)\neq0\},
\end{equation}
\begin{equation}\label{dense open Mij}
\mathcal{M}_{ij}:=\{\pi(B): B\in m_{P,\mathrm{inj}}^{-1}(0)\text{ and }\lambda_i(B)\neq0,\ \lambda_j(B)\neq0,\ \lambda_i(B)+\lambda_j(B)\neq0\}.
\end{equation}
Since the functions $\lambda_i$ and $\lambda_j$ are $\SL_{\mathbf{v}}$-invariant, they descend to regular functions on $m_{P,\mathrm{inj}}^{-1}(0)//\SL_{\mathbf{v}}$. The sets $\mathcal{M}_i$ and $\mathcal{M}_{ij}$ are the corresponding principal open subsets. Similarly, define $\mathcal{M}_i^{\sigma}$ and $\mathcal{M}_{ij}^{\sigma}$ using $m_{P^\sigma,\mathrm{inj}}^{-1}(0)$, where $P^\sigma$ is the standard parabolic whose Levi subgroup is $L^\sigma$, for $\sigma\in S_k$. Denote by $\Omega_i$ and $\Omega_{ij}$ the preimages of $\mathcal{M}_i$ and $\mathcal{M}_{ij}$ inside $m_{P,\mathrm{inj}}^{-1}(0)$ under $\pi$. Similarly, define $\Omega_i^{\sigma}$ and $\Omega_{ij}^{\sigma}$.

Observe that, because $\lambda_i(B)\neq0$ implies $\operatorname{in}_i(B)$ is surjective, the open subset $\Omega_i$ lies in $\Lambda_{P,i}^{\mathrm{surj}}$, and the same holds after replacing $P$ by any $P^\sigma$. Hence the projections $p:Z_i\rightarrow \Lambda_{P,i}^{\mathrm{surj}}$ and $p':Z_i\rightarrow \Lambda_{P',i}^{\mathrm{inj}}$ restrict over the open subsets defined above, and remain principal $\SL_{d_i'}$- and $\SL_{d_i}$-bundles, respectively. Thus the induced geometric reflection map, which we continue to denote by $\psi_i$, restricts to isomorphisms $\psi_i:\mathcal{M}_i\rightarrow\mathcal{M}_i'$ and $\psi_i:\mathcal{M}_{ij}\rightarrow\mathcal{M}_{ij}'$, where the primes refer to the corresponding open subsets for $P'$.

Next, we define $Z_{iji}$ and $Z_{jij}$ as in Maffei:

\begin{defn}
\begin{align*}
Z_{iji}:=\{(s,s''')&\in \Omega_{ij}\times\Omega_{ij}^{\sigma_i\sigma_j\sigma_i}: \exists s'\in m_{P^{\sigma_i},\mathrm{inj}}^{-1}(0) \text{ and } \exists s''\in m_{P^{\sigma_j\sigma_i},\mathrm{inj}}^{-1}(0) \text{ such that }\\
&(s,s')\in Z_i^{\tau_i}(P,P^{\sigma_i}),\quad (s',s'')\in Z_j^{\tau_j}(P^{\sigma_i},P^{\sigma_j\sigma_i}),\quad (s'',s''')\in Z_i^{\tau_i'}(P^{\sigma_j\sigma_i},P^{\sigma_i\sigma_j\sigma_i})\}.
\end{align*}
\end{defn}

\noindent We define $Z_{jij}$ by interchanging the roles of $i$ and $j$. We have suppressed the dependence of the torsion coefficients on the corresponding pairs of parabolics, but we will keep track of it below. Braid relation~\eqref{braid relation} is equivalent, on the dense open subset under consideration, to showing $Z_{iji}=Z_{jij}$.

Suppose $P$ is the standard parabolic associated to the quiver with dimension vector $\mathbf{v}=(d_1,\dots,d_{k-1})$. Introduce the notation
\[
\mathcal{M}(d_{i-1},d_i,d_j,d_{j+1}):=\operatorname{Spec}\big(\C[m_{P,\mathrm{inj}}^{-1}(0)]^{\SL_{\mathbf{v}}}\big).
\]
In this notation, we record only the four vertices involved in applying the reflection functors $\Phi_i$ and $\Phi_j$. Next, we write $\sigma_i$ in place of the reflection functor $\Phi_i(P,P')$. Consider the following diagram:
\begin{equation}\label{braid-hexagon}
\begin{tikzcd}
& \mathcal M(d_{i-1},d_i,d_j,d_{j+1}) \arrow[ld,"\sigma_i"'] \arrow[rd,"\sigma_j"] & \\
\mathcal M(d_{i-1},d_i',d_j,d_{j+1}) \arrow[d,"\sigma_j"'] &&
\mathcal M(d_{i-1},d_i,\widetilde d_j,d_{j+1}) \arrow[d,"\sigma_i"]\\
\mathcal M(d_{i-1},d_i',d_j',d_{j+1}) \arrow[rd,"\sigma_i"'] &&
\mathcal M(d_{i-1},\widetilde d_i,\widetilde d_j,d_{j+1}) \arrow[ld,"\sigma_j"]\\
& \mathcal M(d_{i-1},\widetilde d_i,d_j',d_{j+1}) &
\end{tikzcd}
\end{equation}
The reflection $\sigma_i$ changes $d_i$, while $\sigma_j$ changes $d_j$. Thus
\begin{align*}
    d_i' &= d_{i-1}+d_j-d_i,
    &
    d_j' &= d_i'+d_{j+1}-d_j
          = d_{i-1}+d_{j+1}-d_i, \\
    \widetilde d_j &= d_i+d_{j+1}-d_j,
    &
    \widetilde d_i &= d_{i-1}+\widetilde d_j-d_i
                   = d_{i-1}+d_{j+1}-d_j.
\end{align*}
A priori, the bottom vertex $\mathcal M(d_{i-1},\widetilde d_i,d_j',d_{j+1})$ should involve double primes and double tildes. However, these can be omitted because
\begin{align*}
d_i''&=d_{i-1}+d_j'-d_i'
      =d_{i-1}+d_{j+1}-d_j
      =\widetilde d_i,\\
\widetilde{\widetilde d}_j&=
\widetilde d_i+d_{j+1}-\widetilde d_j
=d_j'.
\end{align*}

In what follows, since we only need to verify an equality of torsions for maps defined by Maffei, we suppress the maps from the notation and refer the reader to Maffei for their definition. For the reader's convenience, we also recall the following key computation of torsion:

\begin{lemma}[{\cite[Lemma~3.9.2 and Corollary~3.9.3]{Wang2021GelfandGraev}}]\label{torsion of maffei SES}
Suppose
\begin{align}
    0\rightarrow V\rightarrow W\oplus U\rightarrow Y\rightarrow 0,\\
    0\rightarrow U\rightarrow X\oplus Y\rightarrow Z\rightarrow 0
\end{align}
are short exact sequences of vector spaces, with maps as in Maffei's Lemmas~40 and~41 \cite{Maffei2002RemarkQuiverWeyl}, and with torsions $\tau(V,Y)$ and $\tau(U,Z)$, respectively. Then the induced sequence
\begin{equation}
    0\rightarrow V\rightarrow W\oplus X\oplus Y\rightarrow Y\oplus Z\rightarrow 0
\end{equation}
is short exact and has torsion $\tau(V,Y)\tau(U,Z)$. Equivalently, after cancelling the common $Y$-term, the sequence
\begin{equation}
    0\rightarrow V\rightarrow W\oplus X\rightarrow Z\rightarrow 0
\end{equation}
is short exact and has torsion $(-1)^{d_Yd_Z+d_Y}\tau(V,Y)\tau(U,Z)$, where $d_Y=\dim Y$ and $d_Z=\dim Z$.
\end{lemma}

Now, using Lemma~\ref{torsion of maffei SES}, we may rewrite the condition $(s,s''')\in Z_{iji}$ in terms of the short exact sequences corresponding to the three edges on the left-hand side of the hexagon \eqref{braid-hexagon}. These short exact sequences are also listed in \cite[p.~679, item~7]{Maffei2002RemarkQuiverWeyl}.

For $r=i$ or $r=j=i+1$, define
\begin{align*}
    V_r&:=\C^{d_r},&
    V_r'&:=\C^{d_r'},&
    V_r''&:=\C^{d_r''},&
    \widetilde{V}_r&:=\C^{\widetilde d_r},&
    \widetilde{\widetilde{V}}_r&:=\C^{\widetilde{\widetilde d}_r}.
\end{align*}
For the specific pair of indices $i,j=i+1$, define
\[
R_i:=V_{i-1},\qquad R_j:=V_{j+1}\simeq V_{i+2}.
\]
These represent the left and right neighboring vertices of the reflected vertices, respectively.

Suppose $(s,s''')\in Z_{iji}$. Choose $s'$ and $s''$ as in the definition of $Z_{iji}$.

\noindent\textbf{(1)} Since $(s,s')\in Z_i^{\tau_i}(P,P^{\sigma_i})$, condition (C2) in Definition~\ref{zipper} gives a short exact sequence
\begin{equation}\label{SES s sprime}
0\rightarrow V_i'\rightarrow R_i\oplus V_j\rightarrow V_i\rightarrow 0
\end{equation}
whose torsion is
\begin{equation}\label{torsion s sprime}
\tau(V_i',V_i)=\tau_i(P,P^{\sigma_i}).
\end{equation}

\noindent\textbf{(2)} Since $(s',s'')\in Z_j^{\tau_j}(P^{\sigma_i},P^{\sigma_j\sigma_i})$, condition (C2) gives a short exact sequence
\begin{equation}\label{SES sprime sdoubleprime}
0\rightarrow V_j'\rightarrow R_j\oplus V_i'\rightarrow V_j\rightarrow 0
\end{equation}
whose torsion is
\begin{equation}\label{torsion sprime sdoubleprime}
\tau(V_j',V_j)=\tau_j(P^{\sigma_i},P^{\sigma_j\sigma_i}).
\end{equation}
Now, applying Lemma~\ref{torsion of maffei SES} to Equations~\eqref{SES sprime sdoubleprime} and~\eqref{SES s sprime}, we obtain a new short exact sequence
\begin{equation}\label{SES sprime sdoubleprime simplified}
0\rightarrow V_j'\rightarrow R_j\oplus R_i\rightarrow V_i\rightarrow 0
\end{equation}
whose torsion is
\begin{equation}\label{torsion sprime sdoubleprime simplified}
(-1)^{d_id_j+d_j}\tau(V_i',V_i)\tau(V_j',V_j)
=
(-1)^{d_id_j+d_j}\tau_i(P,P^{\sigma_i})\tau_j(P^{\sigma_i},P^{\sigma_j\sigma_i}).
\end{equation}

\noindent\textbf{(3)} Since $(s'',s''')\in Z_i^{\tau_i}(P^{\sigma_j\sigma_i},P^{\sigma_i\sigma_j\sigma_i})$, condition (C2) gives a short exact sequence
\begin{equation}\label{SES sdoubleprime stripleprime}
0\rightarrow V_i''\rightarrow R_i\oplus V_j'\rightarrow V_i'\rightarrow 0
\end{equation}
whose torsion is
\begin{equation}\label{torsion sdoubleprime stripleprime}
\tau(V_i'',V_i')=\tau_i(P^{\sigma_j\sigma_i},P^{\sigma_i\sigma_j\sigma_i}).
\end{equation}
Now, applying Lemma~\ref{torsion of maffei SES} to Equations~\eqref{SES sdoubleprime stripleprime} and~\eqref{SES sprime sdoubleprime}, we obtain a new short exact sequence
\begin{equation}\label{SES sdoubleprime stripleprime simplified}
0\rightarrow V_i''\rightarrow R_i\oplus R_j\rightarrow V_j\rightarrow 0
\end{equation}
whose torsion is
\begin{equation}\label{torsion sdoubleprime stripleprime simplified}
(-1)^{d_i'd_j+d_i'}\tau(V_j',V_j)\tau(V_i'',V_i')
=
(-1)^{d_i'd_j+d_i'}\tau_j(P^{\sigma_i},P^{\sigma_j\sigma_i})\tau_i(P^{\sigma_j\sigma_i},P^{\sigma_i\sigma_j\sigma_i}).
\end{equation}

\noindent\textbf{(4)} Since $(s,\widetilde{s})\in Z_j^{\tau_j}(P,P^{\sigma_j})$, condition (C2) gives a short exact sequence
\begin{equation}\label{SES s stilde}
0\rightarrow \widetilde{V}_j\rightarrow R_j\oplus V_i\rightarrow V_j\rightarrow 0
\end{equation}
whose torsion is
\begin{equation}\label{torsion s stilde}
\tau(\widetilde{V}_j,V_j)=\tau_j(P,P^{\sigma_j}).
\end{equation}

\noindent\textbf{(5)} Now, we use the lemma of Maffei on p.~680 of \cite{Maffei2002RemarkQuiverWeyl} to construct
$A\in m_{P^{\sigma_i\sigma_j},\mathrm{inj}}^{-1}(0)$ and a short exact sequence
\begin{equation}\label{SES stilde A}
0\rightarrow \widetilde{V}_i\rightarrow R_i\oplus \widetilde{V}_j\rightarrow V_i\rightarrow 0
\end{equation}
whose torsion we denote by $\tau(\widetilde{V}_i,V_i)$. To conclude that
$(\widetilde{s},A)\in Z_i^{\tau_i}(P^{\sigma_j},P^{\sigma_i\sigma_j})$, it remains only to check that this torsion agrees with the prescribed torsion in condition (C2) of Definition~\ref{zipper}, namely that
\begin{equation}\label{torsion stilde A desired}
\tau(\widetilde{V}_i,V_i)=\tau_i(P^{\sigma_j},P^{\sigma_i\sigma_j}).
\end{equation}

Now, applying Lemma~\ref{torsion of maffei SES} to Equations~\eqref{SES stilde A} and~\eqref{SES s stilde}, we obtain a short exact sequence
\begin{equation}\label{SES stilde A simplified}
0\rightarrow \widetilde{V}_i\rightarrow R_i\oplus R_j\rightarrow V_j\rightarrow 0.
\end{equation}
Its torsion is
\begin{equation}\label{torsion stilde A simplified}
(-1)^{d_id_j+d_i}
\tau(\widetilde{V}_i,V_i)\tau(\widetilde{V}_j,V_j)
=
(-1)^{d_id_j+d_i}
\tau(\widetilde{V}_i,V_i)\tau_j(P,P^{\sigma_j}).
\end{equation}
On the other hand, the short exact sequence \eqref{SES stilde A simplified} is precisely the short exact sequence \eqref{SES sdoubleprime stripleprime simplified}.\footnote{The dimensions agree because $V_i''=\widetilde V_i$, and Maffei checks that the maps agree as well.} Therefore, comparing \eqref{torsion stilde A simplified} with \eqref{torsion sdoubleprime stripleprime simplified}, we obtain the first equality of torsions:
\begin{equation}\label{First torsion equation}
(-1)^{d_i'd_j+d_i'}
\tau_j(P^{\sigma_i},P^{\sigma_j\sigma_i})
\tau_i(P^{\sigma_j\sigma_i},P^{\sigma_i\sigma_j\sigma_i}) =
(-1)^{d_id_j+d_i}
\tau(\widetilde{V}_i,V_i)
\tau_j(P,P^{\sigma_j}).
\end{equation}

\noindent\textbf{(6)} We wish to show
\begin{equation}\label{desired A stripleprime membership}
(A,s''')\in Z_j^{\tau_j}(P^{\sigma_i\sigma_j},P^{\sigma_j\sigma_i\sigma_j}).
\end{equation}
By Maffei, we have a short exact sequence
\begin{equation}\label{SES A stripleprime}
0\rightarrow V_j'\rightarrow R_j\oplus \widetilde{V}_i\rightarrow\widetilde{V}_j\rightarrow0
\end{equation}
whose torsion we denote by $\tau(V_j',\widetilde{V}_j)$. To conclude \eqref{desired A stripleprime membership}, it remains only to check that this torsion agrees with the prescribed torsion in condition (C2) of Definition~\ref{zipper}, namely that
\begin{equation}\label{torsion A stripleprime desired}
\tau(V_j',\widetilde{V}_j)
=
\tau_j(P^{\sigma_i\sigma_j},P^{\sigma_j\sigma_i\sigma_j}).
\end{equation}
Now, applying Lemma~\ref{torsion of maffei SES} to Equations~\eqref{SES A stripleprime} and~\eqref{SES stilde A}, we obtain a short exact sequence
\begin{equation}\label{SES A stripleprime simplified}
0\rightarrow V_j'\rightarrow R_j\oplus R_i\rightarrow V_i\rightarrow 0.
\end{equation}
Its torsion is
\begin{equation}\label{torsion A stripleprime simplified}
(-1)^{\widetilde d_jd_i+\widetilde d_j}
\tau(V_j',\widetilde V_j)\tau(\widetilde V_i,V_i).
\end{equation}
On the other hand, as in step \textbf{(5)}, the short exact sequence \eqref{SES A stripleprime simplified} is the short exact sequence \eqref{SES sprime sdoubleprime simplified}.\footnote{The dimensions agree because both sequences have source $V_j'$, middle term $R_j\oplus R_i$, and target $V_i$; Maffei checks that the maps agree as well.} Therefore, comparing \eqref{torsion A stripleprime simplified} with \eqref{torsion sprime sdoubleprime simplified}, we obtain the second equality of torsions:
\begin{equation}\label{Second torsion equation}
(-1)^{d_id_j+d_j}
\tau_i(P,P^{\sigma_i})
\tau_j(P^{\sigma_i},P^{\sigma_j\sigma_i}) 
 =
(-1)^{\widetilde d_jd_i+\widetilde d_j}
\tau(V_j',\widetilde V_j)
\tau(\widetilde V_i,V_i).
\end{equation}

It remains to choose the torsion constants $\tau_i(P,P')$, for any $i$ and $i$-swapped parabolics $(P,P')$, appearing in Definition~\ref{zipper} so that Equations~\eqref{First torsion equation} and~\eqref{Second torsion equation} hold after substituting
\[
\tau(\widetilde V_i,V_i)=\tau_i(P^{\sigma_j},P^{\sigma_i\sigma_j})
\qquad\text{and}\qquad
\tau(V_j',\widetilde V_j)=\tau_j(P^{\sigma_i\sigma_j},P^{\sigma_j\sigma_i\sigma_j}).
\]

\begin{defn}\label{torsion choice}
Suppose $P$ is a parabolic whose Levi subgroup has block sizes $(s_1,\dots,s_k)$, with $s_r=d_r-d_{r-1}$. We use the convention $s_{k+1}:=1$. Suppose $(P,P')$ is a pair of $i$-swapped parabolics. Define
\[
\tau_i(P,P'):=(-1)^{\ell_i(P)},
\]
where $\ell_i(P)\in\Z/2\Z$ is defined by
\[
\ell_i(P)\equiv d_{i-1}(s_i+s_{i+2})+s_{i+2}(s_i+s_{i+1})+i\,s_{i+1}(1+s_i)\pmod 2.
\]
\end{defn}

\noindent In particular, $\tau_i(P,P')$ depends on the parity of $i$. If the blocks of $P$ all have the same size, then $\tau_i(P,P')=1$ for every $i$.

\begin{prop}\label{torsion choice satisfies braid signs}
With the choice of torsion constants in Definition~\ref{torsion choice}, Equations~\eqref{First torsion equation} and~\eqref{Second torsion equation} are satisfied after substituting
\[
\tau(\widetilde V_i,V_i)=\tau_i(P^{\sigma_j},P^{\sigma_i\sigma_j})
\qquad\text{and}\qquad
\tau(V_j',\widetilde V_j)=\tau_j(P^{\sigma_i\sigma_j},P^{\sigma_j\sigma_i\sigma_j}).
\]
\end{prop}

\begin{proof}
All computations are in $\mathbb F_2$. Set
\[
a=d_{i-1},\qquad u=s_i,\qquad v=s_{i+1},\qquad w=s_{i+2},\qquad z=s_{i+3},\qquad \epsilon\equiv i\pmod 2,
\]
where, if $i=k-2$, the convention in Definition~\ref{torsion choice} gives $z=s_{k+1}=1$. Thus, locally, $P$ has block sizes $(u,v,w,z)$, with the understanding that $z$ may be the auxiliary boundary value $1$.

We first verify \eqref{First torsion equation}. After the stated substitution, \eqref{First torsion equation} is equivalent to
\[
\ell_j(P^{\sigma_i})
+\ell_i(P^{\sigma_j\sigma_i})
+\ell_i(P^{\sigma_j})
+\ell_j(P)
=
(d_i'-d_i)(d_j+1).
\]
Now $d_i'-d_i=d_{i-1}+d_j-2d_i\equiv d_{i-1}+d_j=a+(a+u+v)=u+v$, while $d_j+1=a+u+v+1$. Hence
\[
(d_i'-d_i)(d_j+1)=(u+v)(a+u+v+1)=a(u+v).
\]
Using the local block orders
\[
P=(u,v,w,z),\quad
P^{\sigma_i}=(v,u,w,z),\quad
P^{\sigma_j}=(u,w,v,z),\quad
P^{\sigma_j\sigma_i}=(v,w,u,z),\quad
P^{\sigma_i\sigma_j}=(w,u,v,z),
\]
we compute
\[
\begin{aligned}
\ell_j(P^{\sigma_i})
+\ell_i(P^{\sigma_j\sigma_i})
+\ell_i(P^{\sigma_j})
+\ell_j(P)
&=
\bigl((a+v)(u+z)+z(u+w)+(\epsilon+1)w(1+u)\bigr)\\
&\quad+
\bigl(a(v+u)+u(v+w)+\epsilon w(1+v)\bigr)\\
&\quad+
\bigl(a(u+v)+v(u+w)+\epsilon w(1+u)\bigr)\\
&\quad+
\bigl((a+u)(v+z)+z(v+w)+(\epsilon+1)w(1+v)\bigr)\\
&=a(u+v).
\end{aligned}
\]
This proves \eqref{First torsion equation}.

We now verify \eqref{Second torsion equation}. After the stated substitution, \eqref{Second torsion equation} is equivalent to
\[
\ell_i(P)
+\ell_j(P^{\sigma_i})
+\ell_i(P^{\sigma_j})
+\ell_j(P^{\sigma_i\sigma_j})
=
(\widetilde d_j-d_j)(d_i+1).
\]
Here $\widetilde d_j-d_j=d_i+d_{j+1}-2d_j\equiv d_i+d_{j+1}=(a+u)+(a+u+v+w)=v+w$, while $d_i+1=a+u+1$. Thus
\[
(\widetilde d_j-d_j)(d_i+1)=(v+w)(a+u+1).
\]
On the other hand,
\[
\begin{aligned}
\ell_i(P)
+\ell_j(P^{\sigma_i})
+\ell_i(P^{\sigma_j})
+\ell_j(P^{\sigma_i\sigma_j})
&=
\bigl(a(u+w)+w(u+v)+\epsilon v(1+u)\bigr)\\
&\quad+
\bigl((a+v)(u+z)+z(u+w)+(\epsilon+1)w(1+u)\bigr)\\
&\quad+
\bigl(a(u+v)+v(u+w)+\epsilon w(1+u)\bigr)\\
&\quad+
\bigl((a+w)(u+z)+z(u+v)+(\epsilon+1)v(1+u)\bigr)\\
&=(v+w)(a+u+1).
\end{aligned}
\]
This proves \eqref{Second torsion equation}. Hence the chosen torsion constants make both torsion identities hold.
\end{proof}

\section{Nonisomorphic Varieties with Isomorphic Cotangent Bundles}\label{sectionNonIso}

In this section, we produce a sufficient, though not necessary, numerical criterion for when the underlying varieties $\SL_n/[P,P]$ and $\SL_n/[P^w,P^w]$ are nonisomorphic while their affinized cotangent bundles are isomorphic. The criterion is obtained by computing the dimension of the Zariski tangent space at the origin of the (usually singular) variety $\overline{\SL_n/[P,P]}^{\mathrm{aff}}$. One consequence is that, for any parabolic $P$ with three blocks, the reflection functor which swaps two differently sized adjacent blocks yields a new parabolic $P^{\sigma_i}$ such that $\SL_n/[P,P]$ and $\SL_n/[P^{\sigma_i},P^{\sigma_i}]$ are nonisomorphic.

We proceed by recalling some classical $\SL_n$-invariant theory.

\begin{lemma}[{\cite[Theorem~4.4.4]{DerksenKemperCIT2}}]\label{k=2 generators}
Let $U=\C^n$. Consider the natural $\SL(U)$-action on
$\Hom(\C^r,U)\oplus \Hom(U,\C^s)$. Let $A:\C^r\rightarrow U$ and
$B:U\rightarrow\C^s$ denote the two tautological maps. Then
\[
\C[\Hom(\C^r,U)\oplus \Hom(U,\C^s)]^{\SL(U)}
\]
is generated by the entries of $BA$, together with the coordinates of
$\wedge^n A$ and $\wedge^n B$. Here $\wedge^n A:\wedge^n\C^r\rightarrow\wedge^n U\simeq\C$ and
$\wedge^n B:\wedge^n U\simeq\C\rightarrow\wedge^n\C^s$, so the coordinates
of $\wedge^n A$ and $\wedge^n B$ are the maximal minors of $A$ and $B$,
respectively.
\end{lemma}

Moreover, Popov--Vinberg \cite{PopovVinbergInvariantTheory} describe the defining relations among these generators. In particular, the defining ideal is generated by quadratic relations, namely the Grassmann--Pl\"ucker relations and their flag analogues. Below, we will only use the consequence that there are no nonzero linear relations among the generators provided by Lemma~\ref{k=2 generators}. 

\begin{thm}\label{zariski tangent space}
Let
\[
V=\bigoplus_{i=1}^{k-1}\Hom(\C^{d_i},\C^{d_{i+1}})
\]
and let $\SL_{\mathbf v}:=\prod_{i=1}^{k-1}\SL_{d_i}$, where $d_1<d_2<\dots<d_k=n$. Let $A_i\in\Hom(\C^{d_i},\C^{d_{i+1}})$ denote the tautological map, for $1\leq i\leq k-1$. Then the algebra of invariants
\[
\C[V]^{\SL_{\mathbf v}}
\]
is generated by the coefficients of
\[
\wedge^{d_i}(A_{k-1}\cdots A_i)\in
\Hom(\wedge^{d_i}\C^{d_i},\wedge^{d_i}\C^n)\simeq \wedge^{d_i}\C^n,
\qquad 1\leq i\leq k-1.
\]
Equivalently, $\C[V]^{\SL_{\mathbf v}}$ is generated by the $d_i\times d_i$ minors of the compositions $A_{k-1}\cdots A_i:\C^{d_i}\rightarrow\C^n$. Moreover, there are
\[
\sum_{i=1}^{k-1}\binom{n}{d_i}
\]
such generators, and the ideal of relations among these generators contains no nonzero linear forms.
\end{thm}

\begin{proof}
The case $k=2$ is Lemma~\ref{k=2 generators}, applied with $U=\C^{d_1}$ and with no vector part. Assume $k\geq 3$. We take invariants successively from right to left. At the last internal vertex $U=\C^{d_{k-1}}$, apply Lemma~\ref{k=2 generators} with $A=A_{k-2}:\C^{d_{k-2}}\rightarrow U$ and $B=A_{k-1}:U\rightarrow\C^n$. Since $d_{k-2}<d_{k-1}$, the term $\wedge^{d_{k-1}}A_{k-2}$ contributes no generators. Thus taking $\SL_{d_{k-1}}$-invariants replaces the last two arrows by their composite $A_{k-1}A_{k-2}:\C^{d_{k-2}}\rightarrow\C^n$, and also contributes the coordinates of $\wedge^{d_{k-1}}A_{k-1}$.

Repeating the same argument at the vertices $\C^{d_{k-2}},\dots,\C^{d_1}$ gives precisely the coordinates of $\wedge^{d_i}(A_{k-1}\cdots A_i)$ for $1\leq i\leq k-1$. Since the groups involved are reductive, taking invariants successively is exact, so these successive generating sets generate $\C[V]^{\SL_{\mathbf v}}$.

The number of generators is $\sum_{i=1}^{k-1}\binom{n}{d_i}$, because $\wedge^{d_i}(A_{k-1}\cdots A_i)$ has $\binom{n}{d_i}$ coordinates after choosing the standard volume form on $\wedge^{d_i}\C^{d_i}$. Finally, the defining ideal among these Pl\"ucker-type generators is generated by quadratic relations, by Popov--Vinberg \cite{PopovVinbergInvariantTheory}. Hence the ideal of relations contains no nonzero linear forms.
\end{proof}

\begin{cor}\label{nonisomorphic base}
Suppose $P$ is a parabolic whose Levi subgroup has block sizes $s_1,\dots,s_k$, with $s_i=d_i-d_{i-1}$. Let $w\in S_k$, and let $P^w$ be the corresponding standard parabolic obtained by permuting the blocks of $P$ by $w$. Let $d_i'$ denote the partial sums of the block sizes of $P^w$. Suppose
\[
\sum_{i=1}^{k-1}\binom{n}{d_i}\neq \sum_{i=1}^{k-1}\binom{n}{d_i'}.
\]
Then the affine varieties $\overline{\SL_n/[P,P]}^{\mathrm{aff}}$ and $\overline{\SL_n/[P^w,P^w]}^{\mathrm{aff}}$ are not isomorphic, while the affine varieties $\overline{T^*(\SL_n/[P,P])}^{\mathrm{aff}}$ and $\overline{T^*(\SL_n/[P^w,P^w])}^{\mathrm{aff}}$ are isomorphic.
\end{cor}

\begin{proof}
By Corollary~\ref{generalized base affine quiver}, we may identify $\overline{\SL_n/[P,P]}^{\mathrm{aff}}$ with $V//\SL_{\mathbf v}$. By Theorem~\ref{zariski tangent space}, the coordinate ring $\C[V]^{\SL_{\mathbf v}}$ has $\sum_{i=1}^{k-1}\binom{n}{d_i}$ generators, and the ideal of relations among these generators contains no nonzero linear forms. Therefore the Zariski tangent space at the origin has dimension
\[
\dim T_0\big(\overline{\SL_n/[P,P]}^{\mathrm{aff}}\big)=\sum_{i=1}^{k-1}\binom{n}{d_i}.
\]
Moreover, the same generators embed $\overline{\SL_n/[P,P]}^{\mathrm{aff}}$ into affine space of this dimension, so every Zariski tangent space of $\overline{\SL_n/[P,P]}^{\mathrm{aff}}$ has dimension at most $\sum_{i=1}^{k-1}\binom{n}{d_i}$. Thus this number is the maximal dimension of a Zariski tangent space on $\overline{\SL_n/[P,P]}^{\mathrm{aff}}$. Applying the same argument to $P^w$, the assumed inequality shows that the two affine varieties have different maximal tangent-space dimensions, and hence cannot be isomorphic. The statement about the isomorphism of the affinized cotangent bundles follows from Theorem~\ref{classical reflection functors}.
\end{proof}

\begin{cor}\label{3 blocks}
Suppose $P$ is a parabolic with three blocks, and let $\sigma\in S_3$ be a simple reflection. Then $\overline{\SL_n/[P,P]}^{\mathrm{aff}}$ and $\overline{\SL_n/[P^\sigma,P^\sigma]}^{\mathrm{aff}}$ are isomorphic if and only if $\sigma$ swaps two adjacent blocks of the same size.
\end{cor}

\begin{proof}
Write the block sizes of $P$ as $(a,b,c)$, so $n=a+b+c$. If $\sigma=\sigma_1$, then $P^\sigma$ has block sizes $(b,a,c)$, and the two tangent-space dimensions from Corollary~\ref{nonisomorphic base} are
\[
\binom{n}{a}+\binom{n}{a+b}
\qquad\text{and}\qquad
\binom{n}{b}+\binom{n}{a+b}.
\]
These are equal if and only if $\binom{n}{a}=\binom{n}{b}$, which happens if and only if $a=b$ or $b=n-a$. The second possibility would force $c=0$, impossible. Hence equality occurs if and only if $a=b$.

If $\sigma=\sigma_2$, then $P^\sigma$ has block sizes $(a,c,b)$, and the corresponding dimensions are
\[
\binom{n}{a}+\binom{n}{a+b}
\qquad\text{and}\qquad
\binom{n}{a}+\binom{n}{a+c}.
\]
These are equal if and only if $\binom{n}{a+b}=\binom{n}{a+c}$, which happens if and only if $a+b=a+c$ or $a+c=n-(a+b)$. The first possibility gives $b=c$, while the second gives $a=0$, impossible. Hence equality occurs if and only if $b=c$.

Thus, if $\sigma$ swaps blocks of different sizes, Corollary~\ref{nonisomorphic base} implies that the affinizations are not isomorphic. If $\sigma$ swaps two equal adjacent blocks, then $P^\sigma=P$, so the affinizations are identical.
\end{proof}

We will return to a particularly interesting three-block case -- namely the parabolic subgroups corresponding to the partition $(n-2)+1+1$ -- in Section~\ref{quadricconex} below.

\section{Examples}\label{ExamplesSection}

\subsection{The Example of $\SL_2/U$}\label{sl2ex}

Let us unravel the intertwiner for $\SL_2$. The diagram for the zipper variety $Z_1$ is
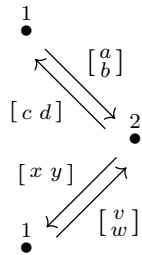
\[
\begin{tikzcd}[ampersand replacement=\&]
\overset{1}{\bullet}
  \arrow[rd, "{\left[\begin{smallmatrix} a \\ b \end{smallmatrix}\right]}"]
\& \\
\& \overset{2}{\bullet}
  \arrow[lu, "{\left[\begin{smallmatrix} c & d \end{smallmatrix}\right]}", shift left=2]
  \arrow[ld, "{\left[\begin{smallmatrix} x & y \end{smallmatrix}\right]}"'] \\
\overset{1}{\bullet}
  \arrow[ru, "{\left[\begin{smallmatrix} v \\ w \end{smallmatrix}\right]}"', shift right=2]
\end{tikzcd}
\]

\noindent Condition (C1) of Definition~\ref{zipper} is vacuous. Condition (C2) of Definition~\ref{zipper} says that the following sequence is exact:
\begin{equation}\label{SL2 C2 exact sequence}
0 \longrightarrow \C
\xrightarrow{\left[\begin{smallmatrix} a \\ b \end{smallmatrix}\right]}
\C^2
\xrightarrow{\left[\begin{smallmatrix} x & y \end{smallmatrix}\right]}
\C
\longrightarrow 0.
\end{equation}
Thus
\begin{equation}\label{SL2 C2 relation}
xa+yb=0.
\end{equation}
The preservation of volume forms says, locally on $x\neq0$ or $y\neq0$, respectively, that
\begin{equation}\label{SL2 torsion local equations}
\begin{bmatrix}a\\b\end{bmatrix}\wedge\begin{bmatrix}x^{-1}\\0\end{bmatrix}=1
\qquad\text{or}\qquad
\begin{bmatrix}a\\b\end{bmatrix}\wedge\begin{bmatrix}0\\ y^{-1}\end{bmatrix}=1.
\end{equation}
Together with \eqref{SL2 C2 relation}, these equations give
\begin{equation}\label{SL2 ab xy relation}
a=y,\qquad b=-x.
\end{equation}
Finally, condition (C3) of Definition~\ref{zipper} says
\begin{equation}\label{SL2 C3 equation}
\begin{bmatrix}v\\w\end{bmatrix}\begin{bmatrix}x& y\end{bmatrix}
=
\begin{bmatrix}a\\b\end{bmatrix}\begin{bmatrix}c&d\end{bmatrix}
-
\begin{bmatrix}ac+bd&0\\0&ac+bd\end{bmatrix}.
\end{equation}
Combining \eqref{SL2 C3 equation} with \eqref{SL2 ab xy relation}, and using exactness of \eqref{SL2 C2 exact sequence}, gives
\begin{equation}\label{SL2 cd vw relation}
v=d,\qquad c=-w.
\end{equation}
In summary, the affine closure of the zipper is
\begin{equation}\label{SL2 zipper graph}
Z=\operatorname{Spec}\big(\C[a,b,c,d,x,y,v,w]/(a-y,b+x,v-d,c+w)\big).
\end{equation}
Thus the reflection functor is precisely the symplectic Fourier transform
\begin{align*}
\Phi:\operatorname{Spec}(\C[a,b,c,d])&\rightarrow\operatorname{Spec}(\C[v,w,x,y]),\\
(a,b,c,d)&\longmapsto (v,w,x,y)=(d,-c,-b,a).
\end{align*}

\subsection{\texorpdfstring{The example of $\SL_3/[P,P]$}{The example of SL3/PP}}\label{1323ex}

We consider the two standard parabolic subgroups of $\SL_3$ with two blocks. The first quiver is
\begin{equation}\label{basicquiversl3}
\begin{tikzcd}
\overset{1}{\bullet} \arrow[r] & \overset{3}{\bullet}.
\end{tikzcd}
\end{equation}
Here $V=\Hom(\C^1,\C^3)$ and the gauge group is $\SL_1=\{1\}$. Thus $V/\!/\SL_1=\A^3$, and the $\mathfrak{sl}_1$-valued moment map is trivial. Hence
\[
m^{-1}(0)/\!/\SL_1=\Hom(\C^1,\C^3)\oplus\Hom(\C^3,\C^1)\cong\A^6.
\]

The reflected quiver is
\begin{equation}\label{reflectedquiversl3}
\begin{tikzcd}
\overset{2}{\bullet} \arrow[r] & \overset{3}{\bullet}.
\end{tikzcd}
\end{equation}
Here $(V,\SL_{\mathbf v})=(\Hom(\C^2,\C^3),\SL_2)$. By Corollary~\ref{generalized base affine quiver}, the affine quotient $V/\!/\SL_2$ identifies with $\overline{\SL_3/[P,P]}^{\mathrm{aff}}$, where $P$ is the $2+1$ parabolic. In this special case, $\overline{\SL_3/[P,P]}^{\mathrm{aff}}\cong\A^3$. Explicitly, for $g=(g_{ij})\in\SL_3$, the map $\SL_3/[P,P]\rightarrow\A^3$ is
\begin{equation}\label{SL3 row map}
g[P,P]\longmapsto
\begin{pmatrix}
(g^{-1})_{31}&(g^{-1})_{32}&(g^{-1})_{33}
\end{pmatrix}
=
\begin{pmatrix}
g_{21}g_{32}-g_{31}g_{22} &
g_{31}g_{12}-g_{11}g_{32} &
g_{11}g_{22}-g_{21}g_{12}
\end{pmatrix}.
\end{equation}
Its image is $\A^3\setminus\{0\}$. Hence Hartogs' Lemma gives $\overline{\SL_3/[P,P]}^{\mathrm{aff}}\cong\A^3$.

We now compare this with the Hamiltonian reduction of the doubled reflected quiver
\begin{equation}\label{doubledsl3quiver}
\begin{tikzcd}
\overset{2}{\bullet} \arrow[r, "A=(a_{ij})", shift left] & \overset{3}{\bullet} \arrow[l, "B=(b_{kl})", shift left]
\end{tikzcd}.
\end{equation}
Let
\[
S:=\{(A,B):BA=\lambda\Id_2\text{ for some }\lambda\in\C\}\subset \Hom(\C^2,\C^3)\oplus\Hom(\C^3,\C^2)
\]
be the shell. The $\SL_2$-action is $(A,B)\mapsto(Ag^{-1},gB)$. Let $p$ be the row vector of signed $2\times2$ minors of $A$, and let $q$ be the column vector of signed $2\times2$ minors of $B$:
\begin{equation}\label{pqcoords}
p=p(A)=
\begin{pmatrix}
a_{21}a_{32}-a_{31}a_{22}
&
a_{31}a_{12}-a_{11}a_{32}
&
a_{11}a_{22}-a_{21}a_{12}
\end{pmatrix},
\quad
q=q(B)=
\begin{pmatrix}
b_{12}b_{23}-b_{13}b_{22}\\
b_{13}b_{21}-b_{11}b_{23}\\
b_{11}b_{22}-b_{12}b_{21}
\end{pmatrix}.
\end{equation}
Both $p$ and $q$ are invariant under this action, since $\det(g)=1$. Let $C:=AB$; its entries are also invariant, and $\tr(C)=\tr(AB)=\tr(BA)=2\lambda$.

\begin{prop}\label{nilporboprop}
The invariant ring $\C[S]^{\SL_2}$ is generated by $\lambda$, the entries of $p$ and $q$, and the entries of $C$. Set
\begin{equation}\label{nilpmatrixsl4}
M(\lambda,C,p,q):=
\begin{pmatrix}
\lambda & p \\
-q & C-\lambda\Id_3
\end{pmatrix}.
\end{equation}
The full ideal of relations among these generators is generated by $2\lambda-\tr(C)$ and the $2\times2$ minors of $M(\lambda,C,p,q)$. In particular, the following identities hold:
\begin{equation}\label{relations23}
2\lambda=\tr(C),\quad pC=0,\quad Cq=0,\quad C^2=\lambda C,\quad qp=\Adj(C),\quad \lambda^2=pq.
\end{equation}
Thus $X:=S/\!/\SL_2$ is isomorphic to the closure of the minimal nilpotent orbit in $\mathfrak{sl}_4$ via the map $(\lambda,C,p,q)\mapsto M(\lambda,C,p,q)$. Using $\mathfrak{so}(6)\cong\mathfrak{sl}_4$, this is also the closure of the minimal nilpotent orbit in $\mathfrak{so}(6)$. In particular, $X$ has dimension $6$ and has an isolated conical singularity at the origin.
\end{prop}

\begin{proof}
By the first fundamental theorem for $\SL_2$, in the form stated in Theorem~\ref{k=2 generators} and recalled from \cite[Theorem~4.4.4]{DerksenKemperCIT2}, the invariant ring $\C[S]^{\SL_2}$ is generated by the entries of $C=AB$, the signed minors $p$ and $q$, and $\lambda=\frac{1}{2}\tr(C)$. The identities \eqref{relations23} follow directly from $BA=\lambda\Id_2$, the definitions of $p$, $q$, and $C$, and the Cauchy--Binet formula.

Let
\[
Y:=\{M\in\mathfrak{sl}_4:\operatorname{rk}(M)\leq1\}.
\]
We first show that the invariant map $X\rightarrow\mathfrak{sl}_4$ defined by \eqref{nilpmatrixsl4} lands in $Y$. On the locus where $A$ has full rank, we have $p\neq0$ and $\operatorname{Im}(A)=\ker(p)$. Since
\[
(C-\lambda\Id_3)A=(AB-\lambda\Id_3)A=A(BA-\lambda\Id_2)=0,
\]
there is a unique column vector $u$ such that $C-\lambda\Id_3=up$. The relations $pC=0$ and $qp=\Adj(C)$ then give $pu=-\lambda$ and $q=-\lambda u$. Consequently,
\[
M(\lambda,C,p,q)
=
\begin{pmatrix}
1\\
u
\end{pmatrix}
\begin{pmatrix}
\lambda&p
\end{pmatrix},
\]
so $M(\lambda,C,p,q)$ has rank at most $1$.

The full-rank locus for $A$ is dense in $S$. Indeed, it is irreducible of dimension $9$, while the locus where $\operatorname{rk}(A)\leq1$ has dimension at most $8$. It follows that all $2\times2$ minors of $M(\lambda,C,p,q)$ vanish identically on $S$. Moreover,
\[
\tr(M(\lambda,C,p,q))=\tr(C)-2\lambda=0.
\]
Thus the invariant map induces a closed immersion $X\hookrightarrow Y$, since $\lambda$, the entries of $p$ and $q$, and the entries of $C$ generate $\C[S]^{\SL_2}$.

The $\SL_2$-action is free on the locus where $A$ has full rank, so $\dim X=9-3=6$. Recall the exceptional isomorphism $\mathfrak{so}(6)\cong\mathfrak{sl}_4$. The variety $Y$ is irreducible of dimension $6$ and is the closure of the minimal nilpotent orbit in $\mathfrak{so}_6$; see \cite[Proposition~3.5]{SlipperSchrodinger}. Hence the closed immersion $X\hookrightarrow Y$ is an isomorphism. Therefore the full ideal of relations is generated by $2\lambda-\tr(C)$ and the $2\times2$ minors of $M(\lambda,C,p,q)$.
\end{proof}

The variety $X$ is singular at the origin; in particular, it is not isomorphic to $\A^6$, which is the Hamiltonian reduction for the quiver \eqref{basicquiversl3}. We now explain how the open part relevant to the reflection functor recovers $\A^6$ after affinization.

Let $\mathcal U\subset S$ be the locus where $A$ has full rank, equivalently where $p\neq0$. This locus is $\SL_2$-stable, and the $\SL_2$-action on $\mathcal U$ is free. Its complement has codimension $1$ in $S$. Indeed, if $\operatorname{rk}(A)\leq1$, then $BA=\lambda\Id_2$ forces $\lambda=0$, and hence $BA=0$. The rank-one locus for $A$ has dimension $4$, and, for each rank-one $A$, the space of matrices $B$ satisfying $BA=0$ has dimension $4$. This gives an $8$-dimensional stratum. The remaining stratum, where $A=0$, has dimension $6$. Thus $S\setminus\mathcal U$ has dimension $8$, and hence codimension $1$ in the $9$-dimensional shell $S$.

\begin{prop}\label{openquotientprop}
The geometric quotient $\mathcal U/\SL_2$ is isomorphic to $(\A^3\setminus\{0\})\times\A^3$. It is a dense open subset of $X$, and its complement has codimension $1$.
\end{prop}

\begin{proof}
Let $X_{p\neq0}\subset X$ be the open locus where $p\neq0$. Define
\begin{equation}\label{open quotient parametrization}
j:(\A^3\setminus\{0\})\times\A^3\rightarrow X,\quad
(p,u)\longmapsto(\lambda,C, p,q),
\end{equation}
where
\begin{equation}\label{lambda,C,p,q}
\lambda=-pu,\quad q=-\lambda u,\quad C=up+\lambda\Id_3.
\end{equation}
Under the identification in \eqref{nilpmatrixsl4}, we have
\[
\begin{pmatrix}
\lambda & p\\
-q & C-\lambda\Id_3
\end{pmatrix}
=
\begin{pmatrix}
1\\
u
\end{pmatrix}
\begin{pmatrix}
\lambda & p
\end{pmatrix}.
\]
This matrix has rank at most $1$ and trace $\lambda+pu=0$, so $j$ lands in $X_{p\neq0}$.

Conversely, suppose that $(\lambda,p,q,C)\in X_{p\neq0}$. The matrix in \eqref{nilpmatrixsl4} has rank at most $1$, and its first row $(\lambda,p)$ is nonzero. Hence there is a unique column vector $u$ such that
\[
C-\lambda\Id_3=up,\quad q=-\lambda u.
\]
The trace equation gives $\lambda+pu=0$, so $\lambda=-pu$. On the patch $p_r\neq0$, the entries of $u$ are given by
\[
u_a=\frac{C_{ar}-\lambda\delta_{ar}}{p_r},
\]
so the inverse is regular. Thus $j$ is an isomorphism onto $X_{p\neq0}$.

The quotient map $S\rightarrow X$ restricts to a surjective morphism $\mathcal U\rightarrow X_{p\neq0}$, since $\mathcal U$ is precisely the inverse image of $X_{p\neq0}$. Its fibers are single $\SL_2$-orbits. Indeed, suppose that $(A,B),(A',B')\in\mathcal U$ have the same invariants. Since $p(A)=p(A')\neq0$, the columns of $A$ and $A'$ are volume-compatible bases of the same plane $\ker(p)$, so there is a unique $g\in\SL_2$ such that $A'=Ag^{-1}$. The equality $AB=A'B'$ and the injectivity of $A$ then give $B'=gB$. Therefore
\[
\mathcal U/\SL_2\cong X_{p\neq0}\cong(\A^3\setminus\{0\})\times\A^3.
\]

Finally, under the realization \eqref{nilpmatrixsl4}, $X\setminus X_{p\neq0}$ is the locus $p=0$. The rank-at-most-one and trace-zero conditions then force $\lambda=0$, so this complement is identified with
\[
\left\{
\begin{pmatrix}
-q & C
\end{pmatrix}\in\Mat_{3\times4}:
\operatorname{rk}\begin{pmatrix}-q&C\end{pmatrix}\leq1,\quad
\tr(C)=0
\right\}.
\]
The rank-at-most-one cone in $\Mat_{3\times4}$ has dimension $6$, and the equation $\tr(C)=0$ cuts out a subvariety of dimension $5$. Since $\dim X=6$, the complement of $\mathcal U/\SL_2$ in $X$ has codimension $1$.
\end{proof}

\begin{remark}
\textup{The complement of $\mathcal U/\SL_2\subset X=S/\!/\SL_2$ has codimension $1$, so a Hartogs-type argument does not recover $X$ from this open subset, even though both $X$ and $\A^6=\Spec\,\Gamma(\mathcal U/\SL_2,\mathcal O)$ are normal. Indeed, $\mathcal U/\SL_2\cong(\A^3\setminus\{0\})\times\A^3$ is an open subset of $\A^6$ whose complement has codimension $3$, while it is also the open subset $X_{p\neq0}\subset X$, whose complement has codimension $1$. Thus $X$ and $\A^6$ are not isomorphic. This explains why the naive reflection functor in Wang's sense cannot give an isomorphism between these affine Hamiltonian reductions, even if one surmounts the normality issues for $\SL$-gauge Hamiltonian reductions.}
\end{remark}

Finally, we explicitly describe how our reflection functor works in this example. The reflection is from the quiver $Q=(1)-[3]$ to the quiver $Q'=(2)-[3]$. For the block sizes $1$ and $2$, Equation~\ref{torsion choice} gives torsion $-1$.\footnote{Since there are 2 blocks, there are no non-trivial braid relations. Hence any choice of nonzero torsion works, but we stick with our convention of Equation \ref{torsion choice}.} The injectivity conditions are $U\neq0$ on the $Q$ side and $\operatorname{rk}(A)=2$ on the $Q'$ side. The additional condition defining $\Lambda_{P,1}^{\mathrm{surj}}$ is $V\neq0$. The zipper diagram is
\begin{equation}\label{SL3 zipper diagram} 
\begin{tikzcd}
\overset{1}{\bullet} \arrow[rd, "U", shift left] &                                                                              \\
                                                 & \overset{3}{\bullet}  \arrow[lu, "V", shift left] \arrow[ld, "B", shift left] \\
\overset{2}{\bullet} \arrow[ru, "A", shift left] &                                                                             
\end{tikzcd}
\end{equation}
Condition (C2) is the short exact sequence
\begin{equation*}
0\longrightarrow\C^2\xrightarrow{A}\C^3\xrightarrow{V}\C\longrightarrow0.
\end{equation*}
With the signed-minor convention in \eqref{pqcoords}, this sequence having torsion $-1$ is precisely $V=-p(A)$. Thus, $V\neq0$ implies that $\operatorname{rk}(A)=2$, while the minor identities give $VA=0$; hence $\operatorname{Im}(A)=\ker(V)$.

We continue to use the scalar $\lambda$ defined by $BA=\lambda\Id_2$. The moment-map equation on the $Q$ side is then $VU=-\lambda$, and condition (C3) is $AB=UV+\lambda\Id_3$. Thus the zipper variety $Z$ is defined by
\begin{equation}\label{SL3 zipper equations}
U\neq0,\quad V\neq0,\quad V=-p(A),\quad VU=-\lambda,\quad BA=\lambda\Id_2,\quad AB=UV+\lambda\Id_3.
\end{equation}

In this case, Theorem~\ref{classical reflection functors} gives us:

\begin{prop}\label{reflectionfunctorsl3prop}
The projection
\[
\pi_{\mathrm{tar}}:Z\rightarrow\{(A,B)\in\mathcal U:B\neq0\},\quad
(U,V,A,B)\longmapsto(A,B)
\]
is an isomorphism. The projection
\[
\pi_{\mathrm{src}}:Z\rightarrow
(\A^3\setminus\{0\})\times((\A^3)^*\setminus\{0\}),\quad
(U,V,A,B)\longmapsto(U,V)
\]
is a principal $\SL_2$-bundle. Consequently, the zipper induces an isomorphism
\[
\Phi:(\A^3\setminus\{0\})\times((\A^3)^*\setminus\{0\})
\xrightarrow{\sim}
\{(A,B)\in\mathcal U:B\neq0\}/\SL_2.
\]
Under the isomorphism of Proposition~\ref{openquotientprop}, the target is the locus where $p\neq0$ and $u\neq0$. Its complement in $\A^6$ has codimension $3$, as does the complement of the source. Hence both sides have affinization $\A^6$.
\end{prop}

In coordinates, Proposition~\ref{openquotientprop} identifies $\mathcal U/\SL_2$ with coordinates $(p,u)$ satisfying Equation \ref{lambda,C,p,q}. On the zipper, $p=-V$, and the identity $C-\lambda\Id_3=UV$ gives $u=U$. We write $\Phi$ for the intertwiner on varieties and $\psi$ for the intertwiner on coordinate rings. Then the coordinate-ring intertwiner is
\begin{equation*}
\psi(p)=-V,\quad \psi(u)= U.
\end{equation*}
After affinization, the corresponding reflection functor is
\begin{equation*}
\Phi:\A^6\longrightarrow\A^6,\quad
(U_1,U_2,U_3,V_1,V_2,V_3)
\longmapsto
(-V_1,-V_2,-V_3,U_1,U_2,U_3).
\end{equation*}
Equivalently, on the invariant generators of the target Hamiltonian reduction,
\begin{equation*}
\psi(\lambda)=-VU,\quad
\psi(p)=-V,\quad
\psi(q)=-(VU)U,\quad
\psi(C)=UV-(VU)\Id_3.
\end{equation*}

Its inverse is given by the same coordinate swap:
\begin{equation*}
\Phi^{-1}:\A^6\longrightarrow\A^6,\quad
(p_1,p_2,p_3,u_1,u_2,u_3)
\longmapsto
(u_1,u_2,u_3,-p_1,-p_2,-p_3).
\end{equation*}
Equivalently, the inverse coordinate-ring map satisfies
\begin{equation*}
\psi^{-1}(U)=u,\quad \psi^{-1}(V)=-p.
\end{equation*}

\subsection{The Quadric Cone and Phantom Cones}\label{quadricconex}

We now return to the three-block parabolics promised in Section~\ref{sectionNonIso}. Let $n\geq 3$, and let $P=P_{1,n-2,1}\subset\SL_n$ be the standard parabolic with block sizes $1+(n-2)+1$. We write the split quadric cone as
\[
C=\{(x_1,\ldots,x_n;y_1,\ldots,y_n)\in\A^{2n}:x_1y_n+x_2y_{n-1}+\cdots+x_ny_1=0\}.
\]

\begin{prop}\label{quadric cone as base affine}
The affine closure $\overline{\SL_n/[P,P]}^{\mathrm{aff}}$ is isomorphic to the quadric cone $C$.
\end{prop}

\begin{proof}
We explicitly describe the invariant functions on $\SL_n/[P,P]$. In block form,
\[
[P,P]=
\begin{pmatrix}
1 & * & *\\
0 & \SL_{n-2} & *\\
0 & 0 & 1
\end{pmatrix}.
\]
For $g\in\SL_n$, right multiplication by $[P,P]$ preserves the first column of $g$; call it $x=(x_1,\dots,x_n)^t$. Likewise, the last row of $g^{-1}$ is preserved; write it as $(y_n,\dots,y_1)$. Since $g^{-1}g=\Id_n$, the last row of $g^{-1}$ pairs trivially with the first column of $g$, so $x_1y_n+x_2y_{n-1}+\cdots+x_ny_1=0.$ Thus we obtain a morphism $\SL_n/[P,P]\rightarrow C$.

Let $C'\subset C$ be the open subset where both $x\neq0$ and $y\neq 0$. Over $C'$, the above map is an isomorphism: given nonzero $x$ and $y$ with $y(x)=0$, choose a basis of $\ker(y)$ whose first vector is $x$, extend it by a vector $e$ with $y(e)=1$, and normalize the determinant to obtain a matrix in $\SL_n$; changing these choices is exactly right multiplication by $[P,P]$. The complement $C\setminus C'$ is $\{x=0\}\cup\{y=0\}$, which has codimension $n-1\geq2$ in $C$. Since $C$ is normal, Hartogs' Lemma gives $\overline{\SL_n/[P,P]}^{\mathrm{aff}}\cong C$.
\end{proof}

\begin{prop}\label{quadric cone cotangent minimal orbit}
The affinization of $T^*\SL_n/[P,P]$ is isomorphic to the closure of the minimal nilpotent orbit of the orthogonal group $\O(Q^+)$, where $Q^+$ is the split quadratic form on $\C^{2n+2}$:
\[
Q^+(x_0,x_1,\ldots,x_n;y_1,\ldots,y_n,y_{n+1})
=
x_0y_{n+1}+x_1y_n+\cdots+x_ny_1.
\]
\end{prop}

\begin{proof}
By Proposition~\ref{quadric cone as base affine}, $\SL_n/[P,P]\cong C'$, where $C'$ is the locus of the quadric cone on which both $x$ and $y$ are nonzero. On the other hand, the affinization of $T^*C^o$, where $C^o=C\setminus\{0\}$, is the closure $\overline{\mathcal O}_{\min}^{D_{n+1}}$ of the minimal nilpotent orbit of $\O(Q^+)$; this is the quasi-classical $F$-moment-descent calculation of \cite[Proposition~3.4]{SlipperSchrodinger}. The complement $C^o\setminus C'$ has dimension $n$, while $C^o$ has dimension $2n-1$. Therefore $T^*C^o\setminus T^*C'$ has codimension $n-1\geq2$ in $T^*C^o$. Since $\overline{\mathcal O}_{\min}^{D_{n+1}}\setminus T^*C^o$ also has codimension at least $2$, Hartogs' Lemma gives
\[
\Gamma(T^*C',\mathcal O)=\Gamma(T^*C^o,\mathcal O)=\C[\overline{\mathcal O}_{\min}^{D_{n+1}}],
\]
and hence $\overline{T^*\SL_n/[P,P]}^{\mathrm{aff}}\cong\overline{\mathcal O}_{\min}^{D_{n+1}}$.
\end{proof}

\begin{remark}\label{phantom cone remark}
The same minimal nilpotent orbit closure also appears for the parabolics whose block sizes are the other permutations of $(1,1,n-2)$. Namely, our reflection functors give isomorphisms
\[
\overline{T^*(\SL_n/[P^{\sigma}_{1,1,n-2},P^{\sigma}_{1,1,n-2}])}^{\mathrm{aff}}
\cong
\overline{\mathcal O}_{\min}^{D_{n+1}},
\qquad \sigma\in S_3.
\]
Thus these spaces give six, not necessarily distinct, polarizations of the same minimal nilpotent orbit closure. Since two block sizes are equal, there are only three distinct standard parabolics, but the full $S_3$ still acts by reflection functors. When a permutation preserves the ordered block sizes, the corresponding reflection functor is an automorphism of the same affinized cotangent bundle.

For the middle parabolic $P_{1,n-2,1}$, the transposition exchanging the first and last blocks preserves the ordered block sizes. Under the identification of Proposition~\ref{quadric cone cotangent minimal orbit}, the resulting automorphism of $\overline{\mathcal O}_{\min}^{D_{n+1}}$ is the quasi-classical quadric Fourier transform. Equivalently, it is induced by conjugation by the Weyl element $w_0$ in the conformal group $\O(Q^+)$; this is the quasi-classical counterpart of the quadric Fourier transform on differential operators described in \cite[Section~3.7]{SlipperSchrodinger}.

The two extreme parabolics $P_{n-2,1,1}$ and $P_{1,1,n-2}$ induce isomorphic varieties $\SL_n/[P,P]$ via the inverse-transpose automorphism of $\SL_n$, followed by conjugation by the anti-diagonal permutation matrix. We call these two varieties the \textit{phantom cones}. Unlike the middle variety $\SL_n/[P_{1,n-2,1},P_{1,n-2,1}]$, which is a quadric hypersurface $C\subset\A^{2n}$, the phantom cones embed in
\[
\A^n\oplus\wedge^2\A^n\simeq\A^{n+\binom{n}{2}},
\]
with coordinates $(x,\omega)$, and their defining ideal is generated by the incidence quadrics $x\wedge\omega=0$, giving $\binom{n}{3}$ equations, together with the Pl\"ucker quadrics $\omega\wedge\omega=0$, giving $\binom{n}{4}$ equations. For $n\geq4$, these varieties are not complete intersections: they have dimension $2n-1$, hence codimension $\binom{n}{2}-n+1$ in $\A^{n+\binom{n}{2}}$, while the defining ideal already has $\binom{n}{3}+\binom{n}{4}$ independent minimal quadratic generators. In particular, (cf. Corollary~\ref{3 blocks}) for $n\geq4$, the phantom cones are not isomorphic to the middle quadric cone. Thus the middle quadric cone and the two phantom cones are distinct singular Lagrangian subvarieties which all live inside the same minimal nilpotent orbit closure.
\end{remark}

\begin{remark}\label{basic affine triality remark}
For $n=3$, the block sizes are $(1,1,1)$, so all six reflection functors are automorphisms of $\overline{T^*(\SL_3/U)}^{\mathrm{aff}}.$
This is precisely the quasi-classical Gelfand--Graev action in type $A_2$. Jia proves that $\overline{T^*(\SL_3/U)}^{\mathrm{aff}}$ is the closure of the minimal nilpotent orbit in $\mathfrak{so}_8$, and that the quasi-classical Gelfand--Graev $S_3$-action is the restriction of Cartan's triality action on $\mathfrak{so}_8$ \cite{JiaAffineClosureSLnU}. In this sense, the $n=3$ case, and the related triality symmetry, is a highly special case of the quadric-and-phantom-cone picture above.
\end{remark}

\begin{remark}\label{LSS quasiclassical remark}
This example is also the quasi-classical analogue of the phenomenon observed by Levasseur--Smith--Stafford. In Section~6.1 of \cite{LevasseurSmithStafford1989Joseph}, they exhibit two nonisomorphic singular affine varieties arising from components of $\overline{\mathcal O}_{\min}\cap\mathfrak n^+$ for $\mathfrak{so}(2n)$ whose rings of differential operators are nevertheless isomorphic. Their common differential-operator algebra is a Joseph quotient of $U(\mathfrak{so}(2n))$. In the present subsection, the cotangent affinizations of the corresponding Braverman--Kazhdan spaces are already isomorphic before quantization. Thus the reflection functors provide a quasi-classical explanation for why the corresponding rings of differential operators should fail to distinguish the underlying varieties.
\end{remark}

\subsection{Non-normal $\SL$-quiver varieties}\label{NonnormalQuiverSection}

We are grateful to Travis Schedler for communicating the following simple example. Consider the quiver $Q=(2)-[2]$. Let
\begin{equation*}
A=
\begin{pmatrix}
a & b\\
c & d
\end{pmatrix},
\qquad
B=
\begin{pmatrix}
e & f\\
g & h
\end{pmatrix}.
\end{equation*}
The moment map $m:T^*(\Hom(\C^2,\C^2))\rightarrow\mathfrak{sl}_2$ is the traceless part of $BA$. Hence $m^{-1}(0)=\{(A,B):BA=\lambda\Id_2\text{ for some }\lambda\in\C\}$. Set
\begin{equation*}
C=AB=
\begin{pmatrix}
ae+bg & af+bh\\
ce+dg & cf+dh
\end{pmatrix}
=:
\begin{pmatrix}
x&y\\
z&w
\end{pmatrix},
\qquad
\alpha=\det(A),
\qquad
\beta=\det(B).
\end{equation*}
By the first fundamental theorem for $\SL_2$, in the form stated in Theorem~\ref{k=2 generators} and recalled from \cite[Theorem~4.4.4]{DerksenKemperCIT2}, the invariant ring is generated by $x,y,z,w,\alpha,\beta$. Write $s=x+w$ and $u=x-w$. On $m^{-1}(0)$ we have $\lambda=\frac{s}{2}$ and
\begin{equation*}
C^2=\lambda C,
\qquad
\alpha\beta=\det(C)=\lambda^2.
\end{equation*}
Moreover, if $C_0:=C-\lambda\Id_2$, then $\alpha C_0=0$ and $\beta C_0=0$. Indeed, multiplying $BA=\lambda\Id_2$ by adjugate matrices gives $\alpha C=\lambda\alpha\Id_2$ and $\beta C=\lambda\beta\Id_2$. Therefore a straightforward elimination gives
\begin{equation*}
\C[m^{-1}(0)]^{\SL_2}
\simeq
\C[s,u,y,z,\alpha,\beta]/I,
\end{equation*}
where
\begin{equation*}
I=
(su,sy,sz,u^2+4yz,\alpha u,\alpha y,\alpha z,\beta u,\beta y,\beta z,4\alpha\beta-s^2)=(u,y,z,4\alpha\beta-s^2)\cap(s,\alpha,\beta,u^2+4yz).
\end{equation*}
Thus $m^{-1}(0)/\!/\SL_2$ is the union of the two quadric surface singularities
\begin{equation*}
V(u,y,z,4\alpha\beta-s^2)
\qquad\text{and}\qquad
V(s,\alpha,\beta,u^2+4yz),
\end{equation*}
which meet only at the origin. In particular, the local ring at the origin has zero divisors: the classes of $s$ and $u$ are nonzero, but $su=0$. Hence $m^{-1}(0)/\!/\SL_2$ is not normal.

Unfortunately, we cannot avoid this issue by restricting our attention to Type-A quivers with strictly increasing dimension vector either:

\begin{prop}\label{six-seven nonnormal}
Let $Q=(6)-[7]$. Then the affine Hamiltonian reduction $m^{-1}(0)/\!/\SL_6$ is not normal. In fact, its coordinate ring is not a domain.
\end{prop}

\begin{proof}
Let $A:\C^6\rightarrow\C^7$ and $B:\C^7\rightarrow\C^6$. Then $m^{-1}(0)=\{(A,B):BA=\lambda\Id_6\text{ for some }\lambda\in\C\}.$
Set $C:=AB$ and $D:=C-\lambda\Id_7$. We have
\[
DA=(AB-\lambda\Id_7)A=A(BA-\lambda\Id_6)=0.
\]
For $1\leq r\leq7$, let $p_r$ be $(-1)^{r+1}$ times the $6\times6$ minor of $A$ obtained by deleting its $r$th row, and set $p=(p_1,\ldots,p_7)$. The cofactor identities give $pA=0$ and, together with $DA=0$,
\[
p_7D_{ij}=D_{i7}p_j,\quad 1\leq i,j\leq7.
\]
Set $h:=D_{13}D_{24}-D_{14}D_{23}$. It follows that
\[
p_7^2h=(D_{17}p_3)(D_{27}p_4)-(D_{17}p_4)(D_{27}p_3)=0\in \C[m^{-1}(0)]^{\SL_6}.
\]

Both $p_7$ and $h$ are $\SL_6$-invariant. If $A$ is the inclusion of the first six coordinate vectors into $\C^7$ and $B$ is the projection onto those coordinates, then $BA=\Id_6$ and $p_7=1$. On the other hand, take $\lambda=0$ and
\[
A(e_1)=e_1,\quad A(e_2)=e_2,\quad A(e_i)=0\text{ for }i\geq3,
\]
with $B(e_3)=e_1$, $B(e_4)=e_2$, and $B(e_i)=0$ otherwise. Then $BA=0$, while $D_{13}=D_{24}=1$ and $D_{14}=D_{23}=0$, so $h=1$. Thus $p_7^2$ and $h$ are nonzero invariant functions whose product is zero.

The invariant ring $\C[m^{-1}(0)]^{\SL_6}$ is a connected graded $\C$-algebra, so if it were normal, it would be a domain. Hence it is not normal.
\end{proof}

This contrasts with the corresponding Hamiltonian reductions for $\GL_{\mathbf{v}}$ gauge groups, which are normal by \cite{CrawleyBoevey2003Normality}.

\bibliography{refs}{}

@misc{SlipperSchrodinger,
  author        = {Slipper, Aaron},
  title         = {Geometrization of the {Schr\"odinger} model for the minimal representation of an even orthogonal group: The de {Rham} setting},
  year          = {2026},
  eprint        = {2604.11023},
  archivePrefix = {arXiv},
  primaryClass  = {math.RT},
  doi           = {10.48550/arXiv.2604.11023}
}

@incollection{LaumonEisenstein,
  author    = {Laumon, G{\'e}rard},
  title     = {Faisceaux automorphes li{\'e}s aux s{\'e}ries d'{Eisenstein}},
  booktitle = {Automorphic Forms, Shimura Varieties, and {$L$}-Functions, Vol. {I} ({Ann Arbor}, {MI}, 1988)},
  series    = {Perspectives in Mathematics},
  volume    = {10},
  pages     = {227--281},
  publisher = {Academic Press},
  address   = {Boston, MA},
  year      = {1990}
}

@article{BravermanGaitsgoryGeometricEisenstein,
  author  = {Braverman, Alexander and Gaitsgory, Dennis},
  title   = {Geometric {Eisenstein} series},
  journal = {Inventiones Mathematicae},
  volume  = {150},
  number  = {2},
  pages   = {287--384},
  year    = {2002},
  doi     = {10.1007/s00222-002-0237-8}
}

@misc{GannonWebsterFunctorialityCoulomb,
  author        = {Gannon, Tom and Webster, Ben},
  title         = {Functoriality of {Coulomb} branches},
  year          = {2025},
  eprint        = {2501.09962},
  archivePrefix = {arXiv},
  primaryClass  = {math.AG}
}

@article{BFGMDrinfeldCompactificationsIC,
  author  = {Braverman, A. and Finkelberg, M. and Gaitsgory, D. and Mirkovi{\'c}, I.},
  title   = {Intersection cohomology of {Drinfeld}'s compactifications},
  journal = {Selecta Mathematica. New Series},
  volume  = {8},
  number  = {3},
  pages   = {381--418},
  year    = {2002},
  doi     = {10.1007/s00029-002-8111-5}
}

@article{BenZviNevinsCuspsDModules,
  author  = {Ben-Zvi, David and Nevins, Thomas},
  title   = {Cusps and {D}-modules},
  journal = {Journal of the American Mathematical Society},
  volume  = {17},
  number  = {1},
  pages   = {155--179},
  year    = {2004},
  doi     = {10.1090/S0894-0347-03-00439-9}
}

@article{BFGMDrinfeldCompactificationsICErratum,
  author  = {Braverman, A. and Finkelberg, M. and Gaitsgory, D. and Mirkovi{\'c}, I.},
  title   = {Erratum to: Intersection cohomology of {Drinfeld}'s compactifications},
  journal = {Selecta Mathematica. New Series},
  volume  = {10},
  pages   = {429--430},
  year    = {2004},
  doi     = {10.1007/s00029-004-0383-5}
}

@incollection{PopovVinbergInvariantTheory,
  author    = {Popov, V. L. and Vinberg, E. B.},
  title     = {Invariant theory},
  booktitle = {Algebraic Geometry {IV}: Linear Algebraic Groups, Invariant Theory},
  editor    = {Parshin, A. N. and Shafarevich, I. R.},
  series    = {Encyclopaedia of Mathematical Sciences},
  volume    = {55},
  pages     = {123--278},
  publisher = {Springer-Verlag},
  address   = {Berlin},
  year      = {1994},
  doi       = {10.1007/978-3-662-03073-8}
}

@misc{GinzburgDerivedSatakeSymplecticDuality,
  author        = {Ginzburg, Victor},
  title         = {Pointwise purity, derived {Satake}, and {Symplectic} duality},
  year          = {2025},
  eprint        = {2508.15958},
  archivePrefix = {arXiv},
  primaryClass  = {math.RT}
}

@article{GannonWilliamsCoulomb,
  author        = {Gannon, Tom and Williams, Harold},
  title         = {Differential operators on the base affine space of {$\mathrm{SL}_{n}$} and quantized {Coulomb} branches},
  journal       = {International Mathematics Research Notices},
  volume        = {2026},
  number        = {2},
  pages         = {rnaf384},
  year          = {2026},
  doi           = {10.1093/imrn/rnaf384},
  eprint        = {2312.10278},
  archivePrefix = {arXiv},
  primaryClass  = {math.RT}
}

@article{GannonCotangentGUP,
  author        = {Gannon, Tom},
  title         = {The cotangent bundle of {$G/U_P$} and {Kostant--Whittaker} descent},
  journal       = {International Mathematics Research Notices},
  volume        = {2025},
  number        = {2},
  pages         = {rnae285},
  year          = {2025},
  doi           = {10.1093/imrn/rnae285},
  eprint        = {2407.16844},
  archivePrefix = {arXiv},
  primaryClass  = {math.RT}
}

@book{DerksenKemperCIT2,
  author    = {Derksen, Harm and Kemper, Gregor},
  title     = {Computational Invariant Theory},
  edition   = {Second enlarged edition},
  series    = {Encyclopaedia of Mathematical Sciences},
  volume    = {130},
  publisher = {Springer},
  address   = {Berlin},
  year      = {2015},
  doi       = {10.1007/978-3-662-48422-7},
  isbn      = {978-3-662-48420-3}
}

@misc{Hsu2026WeylAlgebrasBKSpaces,
  author        = {Hsu, Chun-Hsien},
  title         = {{Weyl} algebras on {Braverman--Kazhdan} spaces},
  year          = {2026},
  eprint        = {2606.01206},
  archivePrefix = {arXiv},
  primaryClass  = {math.RT}
}

@article{DancerKirwanSwann2013,
  author  = {Dancer, Andrew and Kirwan, Frances and Swann, Andrew Francis},
  title   = {Implosion for hyperk{\"a}hler manifolds},
  journal = {Compositio Mathematica},
  volume  = {149},
  number  = {9},
  pages   = {1592--1630},
  year    = {2013},
  doi     = {10.1112/S0010437X13007203}
}

@book{Turaev2002Torsions,
  author    = {Turaev, Vladimir},
  title     = {Torsions of 3-Dimensional Manifolds},
  series    = {Progress in Mathematics},
  volume    = {208},
  publisher = {Birkh{\"a}user},
  address   = {Basel},
  year      = {2002},
  isbn      = {978-3-7643-6918-6}
}

@article{Wang2021GelfandGraev,
  author  = {Wang, Xiangsheng},
  title   = {A new {Weyl} group action related to the quasi-classical {Gelfand--Graev} action},
  journal = {Selecta Mathematica. New Series},
  volume  = {27},
  number  = {3},
  pages   = {38},
  year    = {2021},
  doi     = {10.1007/s00029-021-00655-0}
}

@article{HerbigSchwarzSeaton2024,
  author  = {Herbig, Hans-Christian and Schwarz, Gerald W. and Seaton, Christopher},
  title   = {When does the zero fiber of the moment map have rational singularities?},
  journal = {Geometry \& Topology},
  volume  = {28},
  number  = {7},
  pages   = {3475--3510},
  year    = {2024},
  doi     = {10.2140/gt.2024.28.3475}
}

@article{Kuznetsov1997LaumonResolution,
  author  = {Kuznetsov, Alexander},
  title   = {Laumon's resolution of {Drinfeld}'s compactification is small},
  journal = {Mathematical Research Letters},
  volume  = {4},
  number  = {3},
  pages   = {349--364},
  year    = {1997},
  doi     = {10.4310/MRL.1997.v4.n3.a4}
}

@article{CrawleyBoevey2003Normality,
  author  = {Crawley-Boevey, William},
  title   = {Normality of {Marsden--Weinstein} reductions for representations of quivers},
  journal = {Mathematische Annalen},
  volume  = {325},
  number  = {1},
  pages   = {55--79},
  year    = {2003},
  doi     = {10.1007/s00208-002-0367-8}
}

@article{Lusztig2000QuiverWeyl,
  author  = {Lusztig, G.},
  title   = {Quiver varieties and {Weyl} group actions},
  journal = {Annales de l'Institut Fourier},
  volume  = {50},
  number  = {2},
  pages   = {461--489},
  year    = {2000},
  doi     = {10.5802/aif.1762}
}

@article{Maffei2002RemarkQuiverWeyl,
  author  = {Maffei, Andrea},
  title   = {A remark on quiver varieties and {Weyl} groups},
  journal = {Annali della Scuola Normale Superiore di Pisa - Classe di Scienze},
  series  = {5},
  volume  = {1},
  number  = {3},
  pages   = {649--686},
  year    = {2002},
  url     = {https://www.numdam.org/item/ASNSP_2002_5_1_3_649_0/}
}

@article{Nakajima2003ReflectionFunctors,
  author  = {Nakajima, Hiraku},
  title   = {Reflection functors for quiver varieties and {Weyl} group actions},
  journal = {Mathematische Annalen},
  volume  = {327},
  pages   = {671--721},
  year    = {2003},
  doi     = {10.1007/s00208-003-0467-0}
}

@article{GinzburgKazhdan2022,
  author  = {Ginzburg, Victor and Kazhdan, David},
  title   = {Differential operators on {G/U} and the {Gelfand--Graev} action},
  journal = {Advances in Mathematics},
  volume  = {403},
  pages   = {108368},
  year    = {2022},
  doi     = {10.1016/j.aim.2022.108368}
}

@article{GinzburgRiche2015,
  author  = {Ginzburg, Victor and Riche, Simon},
  title   = {Differential operators on {G/U} and the affine {Grassmannian}},
  journal = {Journal of the Institute of Mathematics of Jussieu},
  volume  = {14},
  number  = {3},
  pages   = {493--575},
  year    = {2015},
  doi     = {10.1017/S1474748014000085}
}

@book{GelfandGraevPiatetskiiShapiro1969,
  author    = {Gel'fand, I. M. and Graev, M. I. and Piatetskii-Shapiro, I. I.},
  title     = {Representation Theory and Automorphic Functions},
  series    = {Generalized Functions},
  volume    = {6},
  publisher = {W. B. Saunders Company},
  year      = {1969},
  note      = {English translation of the 1966 Russian original}
}

@article{JiaAffineClosureSLnU,
  author        = {Jia, Boming},
  title         = {The affine closure of {$T^*(\mathrm{SL}_n/U)$}},
  journal       = {Journal of Lie Theory},
  volume        = {35},
  number        = {1},
  pages         = {83--100},
  year          = {2025},
  eprint        = {2112.08649},
  archivePrefix = {arXiv},
  primaryClass  = {math.RT}
}

@incollection{BerestWilson2004Differential,
  author    = {Berest, Yuri and Wilson, George},
  title     = {Differential isomorphism and equivalence of algebraic varieties},
  booktitle = {Topology, Geometry and Quantum Field Theory},
  editor    = {Tillmann, Ulrike},
  series    = {London Mathematical Society Lecture Note Series},
  volume    = {308},
  pages     = {98--126},
  publisher = {Cambridge University Press},
  address   = {Cambridge},
  year      = {2004},
  doi       = {10.1017/CBO9780511526398.007}
}

@misc{BenZviSakellaridisVenkatesh2024,
  author        = {Ben-Zvi, David and Sakellaridis, Yiannis and Venkatesh, Akshay},
  title         = {Relative {Langlands} duality},
  year          = {2024},
  eprint        = {2409.04677},
  archivePrefix = {arXiv},
  primaryClass  = {math.RT},
  url           = {https://arxiv.org/abs/2409.04677}
}

@misc{GetzEtAl_ModulationGroups_2025,
  author        = {Getz, Jayce R. and Guti{\'e}rrez Terradillos, Armando and Hosseinijafari, Farid and Hu, Bryan and Lee, Seewoo and Slipper, Aaron and Tom{\'e}, Marie-H{\'e}l{\`e}ne and Yao, HaoYun and Zhao, Alan},
  title         = {Modulation groups},
  year          = {2025},
  month         = oct,
  eprint        = {2510.23932},
  archivePrefix = {arXiv},
  primaryClass  = {math.NT},
  doi           = {10.48550/arXiv.2510.23932},
  url           = {https://arxiv.org/abs/2510.23932},
  note          = {arXiv:2510.23932v3 (revised 9 Dec 2025)}
}

@article{LevasseurSmithStafford1989Joseph,
  author  = {Levasseur, Thierry and Smith, S. Paul and Stafford, J. Toby},
  title   = {The minimal nilpotent orbit, the {Joseph} ideal and differential operators},
  journal = {Journal of Algebra},
  volume  = {121},
  number  = {2},
  pages   = {338--358},
  year    = {1989}
}

@misc{TaoLocalBernsteinLebesgue,
  author        = {Tao, Terence},
  title         = {Local {Bernstein} theory, and lower bounds for {Lebesgue} constants},
  year          = {2026},
  eprint        = {2603.21453},
  archivePrefix = {arXiv},
  primaryClass  = {math.CA},
  doi           = {10.48550/arXiv.2603.21453}
}

@article{BK:basic:affine,
  author  = {Braverman, A. and Kazhdan, D.},
  title   = {On the {Schwartz} space of the basic affine space},
  journal = {Selecta Math. (N.S.)},
  volume  = {5},
  number  = {1},
  pages   = {1--28},
  year    = {1999},
  doi     = {10.1007/s000290050041}
}

@article{BK:normalized,
  author  = {Braverman, A. and Kazhdan, D.},
  title   = {Normalized intertwining operators and nilpotent elements in the {Langlands} dual group},
  journal = {Mosc. Math. J.},
  volume  = {2},
  number  = {3},
  pages   = {533--553},
  year    = {2002},
  note    = {Dedicated to Yuri I. Manin on the occasion of his 65th birthday},
  doi     = {10.17323/1609-4514-2002-2-3-533-553}
}

@preamble{
   "\def\polhk#1{\setbox0=\hbox{#1}{\ooalign{\hidewidth
    \lower1.5ex\hbox{`}\hidewidth\crcr\unhbox0}}} "
}

@article{Getz:Liu:BK,
  author        = {Getz, Jayce R. and Liu, Baiying},
  title         = {A refined {Poisson} summation formula for certain {Braverman--Kazhdan} spaces},
  journal       = {Science China Mathematics},
  volume        = {64},
  number        = {6},
  pages         = {1127--1156},
  year          = {2021},
  doi           = {10.1007/s11425-018-1616-0},
  eprint        = {1707.06091},
  archivePrefix = {arXiv},
  primaryClass  = {math.NT}
}

@article{Ngo:Hankel,
  author  = {Ng{\^o}, B. C.},
  title   = {Hankel transform, {Langlands} functoriality and functional equation of automorphic {$L$}-functions},
  journal = {Jpn. J. Math.},
  volume  = {15},
  number  = {1},
  pages   = {121--167},
  year    = {2020},
  doi     = {10.1007/s11537-019-1650-8}
}
\bibliographystyle{alpha}

\end{document}